\documentclass[10pt,electronic]{amsart}        
       
\usepackage{baskervald}
\usepackage[swedish, english]{babel}
\usepackage[T1]{fontenc}
\usepackage[latin1]{inputenc}
\usepackage{microtype}
\usepackage{amsmath}
\usepackage{xcolor}
\usepackage{enumerate}
\usepackage[pdfusetitle]{hyperref}
\usepackage{xfrac}
\usepackage[noabbrev,capitalize]{cleveref}
\hypersetup{
	colorlinks,
	linkcolor={red!60!darkgray},
	citecolor={blue!50!darkgray},
	urlcolor={blue!80!darkgray}
}
\usepackage{comment}
\usepackage{amsthm}                      
\usepackage{amssymb} 
\usepackage{mathrsfs}
\usepackage{stmaryrd}
\usepackage{mleftright}
\usepackage{mathtools}
\usepackage{thmtools}
\usepackage[shortlabels]{enumitem}
\usepackage{tikz-cd}

\usetikzlibrary{babel}
\usetikzlibrary{decorations.markings}

\usepackage{euscript}
\renewcommand{\mathcal}{\EuScript}

\usepackage[hmargin=3.2cm,vmargin=3.5cm,marginparwidth=2.8cm]{geometry} 

\usepackage{todonotes}

\author{Ronno Das}
\author{Dan Petersen}

\title{The Mumford conjecture (after Bianchi)}

\newcommand{\Q}{\mathbf Q}
\newcommand{\Z}{\mathbf Z}
\newcommand{\R}{\mathbf R}
\newcommand{\C}{\mathbf C}

\newcommand{\fP}{\mathfrak{P}_d}
\newcommand{\fS}{\mathfrak{S}}

\newcommand{\from}{\vcentcolon}

\NewDocumentCommand{\set}{somm}{%
\IfNoValueTF{#2}{\IfBooleanTF{#1}{\{#3 \mid #4\}}{\mleft\{ #3 \mathrel{}\middle\vert\mathrel{} #4 \mright\}}}{\mathopen{#2\{}#3 \mathrel{}#2\vert\mathrel{} #4\mathclose{#2\}}}%
}

\newcommand{\Bra}{\mathsf{Bra}}
\newcommand{\rel}{\,\mathsf{rel}\,}

\newcommand{\fr}{\mathsf{fr}}

\renewcommand{\P}{\mathbb P}
\newcommand{\A}{\mathbb A}
\newcommand{\squaretop}{\sfrac14 \partial}
\newcommand{\squareopen}{\sfrac34 \partial}
\newcommand{\loc}{\mathsf{loc}}
\newcommand{\M}{{M}}
\newcommand{\Mod}{\mathrm{Mod}}

\newtheorem{thm}{Theorem}[section]
\newtheorem{cor}[thm]{Corollary}
\newtheorem{lem}[thm]{Lemma}
\declaretheorem[sibling=thm,title={Proposition},refname={Proposition,Propositions},Refname={Proposition,Propositions}]{prop}

\theoremstyle{definition}
\newtheorem{defn}[thm]{Definition}
\theoremstyle{remark}
\newtheorem{rem}[thm]{Remark}

\numberwithin{equation}{section}
\crefname{subsection}{Subsection}{Subsections}
\Crefname{subsection}{Subsection}{Subsections}

\tikzset{->-/.style={decoration={
  markings,
  mark=at position .5 with {\arrow{>}}},postaction={decorate}}}
  
\begin{document}

 \begin{abstract}
     We give a self-contained and streamlined rendition of Andrea Bianchi's recent proof of the Mumford conjecture using moduli spaces of branched covers. 
 \end{abstract}

 \maketitle

\section{Introduction}

Let $\Mod_g$ and $\Mod_g^1$ denote, respectively, the mapping class group of a closed oriented surface of genus $g$, and of a surface of genus $g$ with one boundary component. There are natural maps $\smash{\Mod_g^1\to \Mod_g}$ (gluing on a disk) and ${\Mod_g^1 \to \Mod_{g+1}^1}$ (gluing on a torus with two holes). The celebrated \emph{Harer stability theorem} \cite{harerstability,boldsen,randalwilliams-resolutions} says that  $\smash{H_k(\Mod_g^1;\Z)\to H_k(\Mod_{g+1}^1;\Z)}$ and $\smash{H_k(\Mod_g^1;\Z)\to H_k(\Mod_g;\Z)}$ are isomorphisms for $k \leq \frac {2g-2} 3 $ and $k \leq \frac {2g} 3$, respectively. 

Knowing Harer's theorem, the natural follow-up question is whether one can calculate the (co)homology in the stable range. Rationally, an answer was suggested by Mumford \cite{mumfordtowards}: stably, the rational cohomology ring should be a polynomial algebra, with generators $\kappa_1,\kappa_2,\kappa_3,\ldots$ in degrees $2,4,6,\dots$, the so-called Miller--Morita--Mumford classes. This became known as the \emph{Mumford conjecture.} Miller \cite{miller} had previously proven that the stable rational cohomology ring would be a free algebra (in the graded sense) with at least one generator in each even degree. 

In terms of algebraic geometry, $\Mod_g$ is the orbifold fundamental group of the moduli space $\M_g$ of smooth genus $g$ curves. One can similarly identify $\Mod_g^1$ with the fundamental group of the moduli space $\M_g^1$ of genus $g$ curves with a marked point and a nonzero tangent vector at that marking. Since both spaces are $K(\pi,1)$ as orbifolds, the Mumford conjecture can also be understood as the determination of the rational cohomology of $\M_g$ (or $\M_g^1$) for $g \gg 0$.

The Mumford conjecture was proved about 20 years later, by methods of homotopy theory, in breakthrough work of Madsen and Weiss \cite{madsenweiss}. The starting point is the following theorem of Tillmann \cite{tillmann-stable}. The disjoint union $\mathcal M := \coprod_g B\Mod_g^1$ can be given the structure of a topological monoid, by means of gluing a pair of pants (or boundary-connected sum), and one may form its \emph{group-completion} $\mathcal M_\infty = \Omega B \mathcal M$. According to the group-completion theorem, the stable (co)homology of the mapping class group coincides with the (co)homology of the base component of $\mathcal M_\infty$. It is easy to see geometrically that the monoid structure on $\mathcal M$ extends to an action on $\mathcal M$  by the operad of \emph{two-dimensional little disks}, which implies that $\mathcal M_\infty$ is a \emph{$2$-fold loop space}.\footnote{For a brief overview of the group-completion theorem, and the connection between iterated loop spaces and algebras over the little disk operad, see Section \ref{groupcompletion section} of this paper.} The surprising theorem of Tillmann is that it is in fact an \emph{infinite loop space}. Madsen and Weiss's proof of the Mumford conjecture goes by identifying precisely this infinite loop space: it is the Thom spectrum of the virtual vector bundle $-\mathcal O(-1)$ over $\C\mathbf P^\infty$. The rational cohomology of the associated infinite loop space is very easy to calculate, and the Mumford conjecture follows. 

Many alternative proofs of the Madsen--Weiss theorem have since appeared, including the more general Galatius--Madsen--Tillmann--Weiss theorem \cite{gmtw} and the Cobordism Hypothesis \cite{luriecobordism}. However, until recently, the only known proof of the Mumford conjecture was via the Madsen--Weiss theorem; that is, by determining explicitly the infinite loop space in Tillmann's theorem. This changed with recent remarkable work of Bianchi, who has given a completely different proof of the Mumford conjecture, spread across a sequence of four papers \cite{bianchi1,bianchi2,bianchi3,bianchi4}. The goal of our paper is to give a streamlined and self-contained account of Bianchi's proof. 

The starting point in Bianchi's work is not Tillmann's infinite loop space structure on $\mathcal M_\infty$, but instead the geometrically obvious $2$-fold loop space structure coming from pair-of-pants multiplication. The first goal is to explicitly identify a geometrically meaningful \emph{$2$-fold delooping} $X$ of $\mathcal M_\infty$ (a space such that $\mathcal M_\infty \simeq \Omega^2 X$), using ideas of Segal and McDuff (the \emph{scanning map}). The second goal is to show that $X$ has the rational homotopy type of a product of Eilenberg--MacLane spaces $\prod_{n=1}^\infty K(\Q,2n)$. This directly implies the Mumford conjecture, since the base component of $\mathcal M_\infty \simeq \Omega^2 X$ is then also rationally homotopy equivalent to $\prod_{n=1}^\infty K(\Q,2n)$.
 
 As a model for how to carry out the first step, consider the following easier problem. Let $D$ be an open two-dimensional disk, and $\mathrm{Conf}_n\,D$ be the configuration space of $n$ distinct unordered points in $D$. Then $\mathcal C = \coprod_n \mathrm{Conf}_n\,D$ admits the structure of an algebra over the operad of two-dimensional little disks, which gives the group-completion $\mathcal C_\infty$ the structure of a $2$-fold loop space. The corresponding $2$-fold delooping was first determined by Segal \cite{segalconfiguration}: $\mathcal C_\infty \simeq \Omega^2 S^2.$ There are by now very many ways to prove this result. A particularly geometric approach goes back to McDuff \cite{mcduff}: one model of the 2-fold delooping is given by the \emph{configuration space of arbitrarily many distinct unordered points in $D$, where points can be annihilated and created along $\partial D$}. Then it suffices to show that this latter space has the homotopy type of a $2$-sphere. To do this, one may ``expand'' configurations radially from the origin, to make all but at most one point in a configuration hit the boundary and be annihilated. This deformation retracts the space of particles-with-annihilation in $D$ to its subspace consisting of configurations of at most one particle. But this subspace is homeomorphic to the one-point compactification of $D$, with the point at infinity corresponding to the empty configuration. Thus $\mathcal C_\infty \simeq \Omega^2 S^2.$
 
 Returning now to the setting of moduli of curves, consider a compact Riemann surface $\Sigma$ of genus $g$, with a marked point $x$. The Riemann--Roch theorem implies that for $d$ large enough with respect to $g$, there exists a  meromorphic function on $\Sigma$ with a unique pole at $x$, of order $d$. A more careful argument in families shows similarly that for $d$ large enough with respect to $g$, there is an affine bundle over $\M_g^1$ which parametrizes smooth pointed genus $g$ curves $(\Sigma,x)$ equipped with a degree $d$ branched cover $\Sigma \to \mathbb P^1$ such that $x$ is a point of total ramification over $\infty$. (We also need to normalize the branched cover by asking it to take a fixed tangent vector at $x$ to a fixed tangent vector at $\infty$, via its ``principal part of order $d$''.) Taking the real blow-up of $\Sigma$ and $\mathbb P^1$ at $x$ and $\infty$, we get a genus $g$ surface with a boundary component, which is a branched cover of a closed disk. The upshot is that we may replace $B\Mod_g^1 \simeq \M_g^1$ with a homotopically equivalent \emph{topological moduli space} parametrizing \emph{branched covers of the closed disk} with prescribed degree, genus, and boundary monodromy. 
 
 The point, now, is that the moduli space of branched covers of the disk\footnote{Moduli spaces of branched covers of Riemann surfaces are classically studied in algebraic geometry under the name \emph{Hurwitz spaces}.} is not so different from the configuration space of points in the disk! Indeed, a branched cover is completely described by its branch locus (a configuration of finitely many points), together with a finite amount of ``monodromy data''. Thus we may think of the moduli-of-surfaces monoid $\mathcal M$ as being rather similar to the monoid $\mathcal C$ of configuration spaces. A key difference between configuration spaces and moduli of branched covers is that branch points can collide, and split apart, when the monodromy allows it: for example, the map $z \mapsto z^n$ has a single branch point (critical value) of index $n$, whereas a generic perturbation $z \mapsto f(z)$, $\deg f=n$, will have $n-1$ simple branch points. Nevertheless, it turns out that the geometric delooping strategy described above works in this setting, too: a $2$-fold delooping is given by the  \emph{moduli space of branched covers of $D$ with any number of branch points, where branch points can be annihilated and created along $\partial D$}.
 
 It remains to study the homotopy type of this $2$-fold delooping. Unlike the case of configuration spaces of points, it does not deformation retract to its subspace of covers with at most one branch point: the problem being that a non-simple branch point can always split apart into multiple branch points under arbitrarily small perturbations. Nevertheless the same radial expansion trick provides a deformation retract to a subspace of particularly ``simple'' covers, namely those branched covers $E\to D$ where $E$ is a disjoint union of disks. This subspace admits an explicit CW-decomposition with only even-dimensional cells, from which we immediately read off its (co)homology groups. Even better, it comes with a geometrically meaningful cellular approximation of the diagonal, which can be used to compute the cup-product in cohomology; the cohomology ring is a polynomial algebra, and hence the space is a product of Eilenberg--MacLane spaces rationally. This proves the Mumford conjecture.

 The sketch above elides many subtleties in the argument. Let us try to rectify the most egregious omissions: 
 \begin{enumerate}
 	\item The space $\M_g^1$ is only equivalent to a space of degree $d$ branched covers of the disk when $d$ is large enough with respect to $g$. Since we want to study what happens as $g\to \infty$, we also need to let the degree of the cover grow arbitrarily. It seems that the most convenient way to deal with this is by working with a complicated topological monoid parametrizing branched covers of any degree, with any boundary monodromy, and any number of connected components of any genera. In particular, its monoid of connected components is certainly not $\mathbf N$. On the other hand, when studying the delooping it will be more convenient to fix the degree $d$. 
 	\item In order to glue branched covers together, they need to be equipped with trivializations along (part of) the boundary of the disk. In the 2-fold delooping there is no such trivialization as part of the data, which forces us to treat the delooping as a \emph{topological stack}: the objects it parametrizes may have nontrivial automorphisms. For example, the trivial degree $d$ cover of the disk has automorphism group $\mathfrak S_d$. 
 	\item Since the $2$-fold delooping is a stack, the CW-decomposition we construct of it is only a decomposition into orbicells. In particular, it can only be easily used to determine the \emph{rational} (co)homology of the delooping. 
 \end{enumerate}

Although the global structure and strategy of the proof given here is the same as Bianchi's, there are many differences in implementation details. To aid the reader in translating between this paper and Bianchi's, we compare the two approaches in the final Section~\ref{sec:bianchi-comparison} of this paper.


 \subsection{Acknowledgements} We are grateful to Andrea Bianchi for useful comments and correspondence. More importantly, Andrea explained his proof of the Mumford conjecture to the second author in the spring of 2022 at the Mittag-Leffler institute, and around the same time he provided several helpful explanations about deloopings of Hurwitz spaces in connection with the writing of \cite{BDPW}; these delooping arguments appear with only small modifications also in \cref{sec:4} of this paper. 

 Many hours worth of conversations with Craig Westerland were crucial in defining correctly the topology on the space of branched covers in \cref{sec:2}.

David Rydh provided useful explanations about gluing of stacks. 

We thank the referee for their careful reading.

Finally, we want to say that this paper owes a lot to Allen Hatcher's exposition \cite{hatcher-madsenweiss} of the proof of the Madsen--Weiss theorem via scanning, following \cite{grw}.

 \section{Group-completion, \texorpdfstring{$E_k$}{E_k}-algebras, and delooping}\label{groupcompletion section}

The purpose of this section is to briefly review the group-completion theorem and related material, most of which goes back to the 1970's and work of Quillen, May, and Segal, among others. This is all well-known among homotopy theorists, but not yet included in textbooks, and less familiar to algebraic geometers and geometric group theorists who may also be interested in the Mumford conjecture. 

\subsection{\texorpdfstring{$E_1$}{E_1}-algebras and the group-completion functor}\label{section about e1-algebras}

Let $M$ be a \emph{topological $E_1$-algebra}, also known as an algebra over the operad of $1$-dimensional little disks. The reader who is not familiar with this notion will not lose much by assuming that $M$ is a topological monoid: every topological monoid is a topological $E_1$-algebra, and every topological $E_1$-algebra is weakly equivalent to a topological monoid. 
If $X$ is a based space then its loop space $\Omega X$ is a topological $E_1$-algebra. The functor $X \mapsto \Omega X$ admits a left adjoint $B$, the \emph{classifying space} functor: 
\begin{equation}\label{eq:first adjunction}\Omega : (\text{based spaces}) \leftrightarrows (\text{topological $E_1$-algebras}) : B\end{equation}
We will prefer to think of \eqref{eq:first adjunction} as an adjunction of $\infty$-categories, but the reader will not lose much by thinking of it as an adjunction between homotopy categories.

If $M$ is a topological $E_1$-algebra, then $\pi_0(M)$ is a monoid. We say that $M$ is \emph{group-like}, or \emph{group-complete}, if $\pi_0(M)$ is in fact a group. The $E_1$-algebra $M$ is group-like if and only if $M \to \Omega B M$ is a weak equivalence, if and only if $M$ is weakly equivalent to a loop space. It follows that the full subcategory of group-like $E_1$-algebras is a reflective subcategory: there is an adjunction
\begin{equation}\label{eq:second adjunction}\text{incl} : (\text{group-like topological $E_1$-algebras}) \leftrightarrows (\text{topological $E_1$-algebras}) : \Omega B\end{equation}
where the reflector $M \to \Omega BM$ --- the \emph{group-completion} functor --- is precisely given by the unit of the adjunction \eqref{eq:first adjunction}. Again we prefer to think of \eqref{eq:second adjunction} as an adjunction of $\infty$-categories.

As just stated, the unit $M \to \Omega BM$ of the adjunction \eqref{eq:first adjunction} is an equivalence if and only if $M$ is group-like. The counit $B \Omega X \to X$ is an equivalence if and only if $X$ is connected. Hence \eqref{eq:first adjunction} restricts to an equivalence of $\infty$-categories between connected based spaces and group-like topological $E_1$-algebras. In particular, if $X$ is connected, then a map $M \to \Omega X$ of topological $E_1$-algebras exhibits $\Omega X$ as the group-completion of $M$ if and only if its adjoint $BM \to X$ is a weak homotopy equivalence. A very modern reference for all of the above is \cite[Theorem 5.2.6.10]{luriehigheralgebra}.

\subsection{The group-completion theorem} The paradigm of \emph{homological stability} is by now ubiquitous in algebra, geometry and topology.  Quillen's \emph{group-completion theorem} provides a link between homological stability and the group-completion operation just described. Before stating the theorem we need some preliminary definitions. 

\begin{defn}If $S$ is any discrete monoid, then we let $ES$ denote the category whose objects are the elements of $S$, and a morphism $u \to v$ is an element $s \in S$ with $us=v$.
\end{defn}\begin{defn}If $R$ is a ring, then a multiplicative subset $S \subset R$ is said to satisfy the \emph{right Ore conditions} if the following two conditions are satisfied: 
	\begin{enumerate}[(i)]
		\item For $r \in R$ and $s \in S$ with $sr = 0$, there exists $t \in S$ with $rt=0$.
		\item For $r \in R$ and $s \in S$, there exists $r' \in R$ and $s' \in S$ with $rs' = sr'$. 
\end{enumerate}\end{defn}
There are many variations of the group-completion theorem in the literature, e.g.\ \cite{mcduffsegal,oscar-groupcompletion,nikolaus}. The following \cref{groupcompletion} is closest to the original version of Quillen, first published in \cite[Appendix~Q]{quillen-groupcompletion}.
\begin{thm}[Group-completion theorem]\label{groupcompletion}Let $M$ be a topological $E_1$-algebra, with connected components $M=\coprod_{s \in S} M_s$. Suppose that $S = \pi_0(M) \subset H_\bullet(M;\Z)$ satisfies the right Ore conditions, and that $ES$ is a filtered category. Then 
	$$ H_\bullet(\Omega_0 BM;\Z) \cong \varinjlim_{s \in ES} H_\bullet(M_s;\Z),$$
	where $\Omega_0 BM$ denotes the base component of $\Omega BM$. 
\end{thm}

\begin{rem}
While this is an extremely useful tool in computing stable homology, the assumptions in the theorem or its conclusion have no bearing on whether $H_\bullet(M_s; \Z)$ stabilizes (as $s \to \infty$, in whatever sense).
\end{rem}


\subsection{Example: algebraic \texorpdfstring{$K$}{K}-theory}

It is easiest to illustrate how \cref{groupcompletion} relates to homological stability by an example. 	Let $A$ be a commutative ring. Let $C$ be the monoidal category of finitely generated projective $A$-modules and their isomorphisms, with monoidal structure given by direct sum. Let $M$ be the classifying space of $C$, i.e.\ the geometric realization of its nerve. Then $M$ is a topological $E_1$-algebra (as is the classifying space of any monoidal category), and 
	$$ M \cong \coprod_{[P]} B\mathrm{Aut}_A(P)$$
	with the coproduct ranging over all isomorphism classes of finitely generated projective $A$-modules. The group-completion $\Omega BM$ is in this case an \emph{infinite loop space}, denoted $\Omega^\infty K(A)$, where $K(A)$ is the \emph{algebraic $K$-theory spectrum} of $R$. This is one way of \emph{defining} algebraic $K$-theory. 
	
	 Since $C$ is \emph{symmetric} monoidal, $\pi_0(M)$ and $H_\bullet(M;\Z)$ are commutative, which implies that $E\pi_0(M)$ is filtered and the Ore conditions are satisfied. Hence by the group-completion theorem one has 
	$$ H_\bullet(\Omega_0^\infty K(A);\Z) \cong \varinjlim_{[P]} H_\bullet(\mathrm{Aut}_A(P);\Z).$$
	Since every finitely generated projective module is a direct summand of a finitely generated free module, free modules are \emph{cofinal} in the colimit on the right, and hence we may also write
	$$ H_\bullet(\Omega_0^\infty K(A);\Z) \cong \varinjlim_{n} H_\bullet(\mathrm{GL}_n(A);\Z).$$
	The upshot is that whenever the ring $A$ is such that the family of discrete groups $\mathrm{GL}_n(A)$ satisfies homological stability (which is often the case), then the stable homology of $\mathrm{GL}_n(A)$ coincides with the homology of the infinite loop space of algebraic $K$-theory of $A$. This is a useful tool for calculating the groups $K_i(A) = \pi_i(K(A))$. 

More generally, the conclusion is that if $M$ is a topological $E_1$-algebra satisfying the conditions of \cref{groupcompletion}, for which the connected components of $M$ satisfy a homological stability theorem --- or, more generally, a cofinal subset of components satisfy homological stability --- then the problem of computing the stable homology is exactly equivalent to the problem of computing the homology of the space $\Omega_0 BM$.

\subsection{\texorpdfstring{$E_k$}{E_k}-algebras}

We denote by $E_k$ the topological operad of $k$-dimensional little disks. The motivating examples of $E_k$-algebras are iterated loop space: if $X$ is a based space, then $\Omega^k X$ is a topological $E_k$-algebra. Up to group-completion, all topological $E_k$-algebras are of this form \cite[Theorem~5.2.6.10]{luriehigheralgebra}.

\begin{thm}[Recognition principle for iterated loop spaces] For all $k \in \mathbb N_{>0}$ there is a functor $B^k$ from topological $E_k$-algebras to based spaces, and a functor $B^\infty$ from topological $E_\infty$-algebras to spectra. If $M$ is a group-like topological $E_k$-algebra, then $M \simeq \Omega^k B^k X$ as $E_k$-algebras. \end{thm}

\begin{rem}
	When $k=1$, the functor $B^k$ is the usual classifying space $B$. The space $B^k M$ is said to be the \emph{$k$-fold delooping} of $M$. \end{rem}

\begin{rem}The recognition principle ``explains'' why in the example of algebraic $K$-theory, the group-completion produces an infinite loop space: the classifying space of a symmetric monoidal category is always an $E_\infty$-algebra. \end{rem}

As stated in the introduction, our main interest will be the topological $E_2$-algebra $M = \coprod_g B\mathrm{Mod}_g^1$ given by the disjoint union of classifying spaces of all mapping class groups. The $E_2$-algebra structure can (with some care) be seen geometrically by means of gluing of surfaces, or algebraically by using that the classifying space of a braided monoidal category is always an $E_2$-algebra. Our goal will be to find a geometric model for the $2$-fold delooping $B^2M$ whose homotopy type we can understand well enough to deduce the Mumford conjecture. 

The standard construction of the functor $B^k$ is by a two-sided bar construction involving the operad $E_k$. This is not the tool we will use to understand $B^2 M$. Our preferred model for $M$ (or rather an enlargement $M$ of it, involving branched covers of disks) will come with two suitably compatible $E_1$-algebra structures. The $2$-fold delooping of $M$ will then be constructed by applying twice the usual classifying space functor for topological monoids. This is philosophically an avatar of the Dunn--Lurie additivity theorem \cite[Theorem~5.1.2.2]{luriehigheralgebra}, although the additivity theorem will play no role in our arguments.

\section{Moduli of curves and moduli of branched covers}\label{sec:1}

\subsection{The space of branched covers of the line}\label{sec: space of branched covers of the line}

Let $\lambda \vdash d$ be a partition, i.e.\ $d=\sum_i \lambda_i$, and $g \in \Z_{\geq 0}$. Following standard notation in Hurwitz theory, we let $M_g(\P^1,\lambda)$ denote the algebro-geometric moduli stack parametrizing pairs $(C,f)$ where $C$ is a smooth projective curve of genus $g$, and $f : C \to \P^1$ is a morphism of degree $d$ such that the ramification profile over $\infty$ is given by the partition $\lambda$. 

Suppose that $x \in C$ is a degree $k$ ramification point of a map $h:C\to D$ of smooth curves. The principal part of order $k$ of $h$ at $x$ may be considered as a degree $k$ covering map $T_xC\setminus \{0\} \to T_{h(x)}D\setminus \{0\}$ on tangent spaces, as in \cite[\S15.14]{deligne-3pts}. For a degree $d$ cover $f:C \to \P^1$, there is thus an induced $d$-fold covering space 
$$ f_\infty : \coprod_{x \in f^{-1}(\infty)} (T_xC\setminus \{0\}) \to T_\infty\P^1\setminus\{0\}.$$
Fix once and for all the choice of a ``standard'' tangent vector $\Vec v$ at $\infty$ of $\P^1$. Denote by $M_g^\fr(\P^1,\lambda)$ the finite cover of $M_g(\P^1,\lambda)$ which parametrizes the additional choice of a total ordering of the $d$ preimages $f^{-1}_\infty(\Vec v)$. 
The monodromy of a small loop around $\infty$ permutes the set $f_\infty^{-1}(\Vec v)$, which means that $M_g^\fr(\P^1,\lambda)$ breaks up into connected components indexed by permutations of cycle-type $\lambda$:
$$ M_g^\fr(\P^1,\lambda) = \coprod_{\substack{\pi \in \mathfrak S_d \\ [\pi]=\lambda}} M_g^\fr(\P^1,\lambda)_\pi.$$
In the above formula, $[\pi]=\lambda$ means that $\lambda$ is the conjugacy class of $\pi$, where we identify conjugacy classes in $\mathfrak S_d$ with partitions of $d$. 
Then $M_g(\P^1,\lambda) = [M_g^\fr(\P^1,\lambda)/\mathfrak S_d]$, where $[-/\mathfrak S_d]$ denotes the stack quotient by the action of $\mathfrak S_d$, with the action of $\mathfrak S_d$ permuting the connected components via conjugation. The space $ M_g^\fr(\P^1,\lambda)$ is a scheme, not just a Deligne--Mumford stack. To see that each $ M_g^\fr(\P^1,\lambda)_\pi$ is connected, one may consider the dense open subset parametrizing covers with only simple branching outside $\infty$. This subset is a finite-sheeted covering space of a configuration space of distinct unordered points on $\mathbb A^1$; it follows from \cite[Theorem~1]{kluitmann} that the monodromy group is transitive, so that the covering space is connected. The case where there is no ramification over $\infty$, i.e. $\lambda = (1^d),$ is the classical Clebsch--L\"uroth--Hurwitz theorem.

Denote by $\alpha$ the permutation $(12\ldots d)$.  Let $\mathrm{Pol}_d $ be the space of monic degree $d$ polynomials; it is isomorphic as a scheme to $\A^d$. We write $\mathbb G_m$ and $\mathbb G_a$ for the multiplicative and additive group; over $\C$ these are $(\C^\times,\cdot)$ and $(\C,+)$, respectively. Then $\mathbb G_a$ acts on $\mathrm{Pol}_d$ by precomposition; as does the group $\mu_d \ltimes \mathbb G_a$, considered as a subgroup of the group of affine transformations $\mathbb G_m \ltimes \mathbb G_a$, where $\mu_d$ denotes the group of $d$th roots of unity. 

\begin{prop}\label{depressed-polynomials}
    $M_0^\fr(\P^1,(d))_\alpha \cong \mathrm{Pol}_d/\mathbb G_a \cong \mathbb A^{d-1}$, where $(d)$ is the partition of $d$ with a single block. Similarly, $M_0(\P^1,(d)) \cong [\mathrm{Pol}_d/(\mu_d \ltimes \mathbb G_a)] \cong [\mathbb A^{d-1}/\mu_d]$.
\end{prop}

\begin{proof}
    A branched cover $\P^1 \to \P^1$ with total ramification over $\infty$ is a polynomial. Asking it to preserve the tangent vector $\Vec v$ at $\infty$ amounts to asking the polynomial to be monic. Since we have no preferred coordinate system on the covering curve, we must take the quotient by the subgroup of $\mathrm{Aut}(\P^1)$ fixing $\infty$ and $\Vec v$, which is $\mathbb G_a$. 
    
    The second assertion follows by taking the stack quotient by the subgroup of $\mathfrak S_d$ stabilizing the connected component $M_0^\fr(\P^1,(d))_\alpha$; this group is just the cyclic subgroup of $\mathfrak S_d$ generated by $\alpha$. 
\end{proof}

\subsection{Points over the complex numbers}\label{subsec:points over C}

Recall that if $Y$ is a smooth curve, then a branched cover $f:X \to Y$ is uniquely determined by its branch locus $\Delta \subset Y$ and the covering space $f^{-1}(Y \setminus \Delta) \to (Y \setminus \Delta)$. (Moreover, a finite-sheeted topological covering space of a Zariski open subset of $Y$ always extends uniquely to an \emph{algebraic} branched cover, by Riemann's existence theorem.) Thus a geometric point of $M_g^\fr(\P^1,\lambda)_\pi$ can be uniquely described by specifying a finite subset $\Delta$ of $\A^1$, a degree $d$ covering space of $\A^1 \setminus \Delta$, and a trivialization of the cover along $\Vec v$. Equivalently, this amounts to specifying a finite subset $\Delta \subset \A^1$, and a homomorphism
$$ \rho : \pi_1(\A^1 \setminus \Delta, \Vec{v}) \to \mathfrak S_d,$$
where $\Vec{v}$ is considered as a tangential basepoint \cite{deligne-3pts}. Let us now characterize which $\rho$ will correspond to actual points of $M_g^\fr(\P^1,\lambda)_\pi$.

\begin{defn}
    For a permutation $\pi \in \mathfrak S_d$, let 
\[ N(\pi) := d - \# \{\text{cycles in }\pi\} = \min \{ k : \pi \text{ is a product of } k \text{ transpositions}\}.\]
\end{defn}
\begin{rem}
    Note that $N(-)$ is a class function. In the combinatorics literature it is called the \emph{absolute length} of $\pi$. 
\end{rem}

If we choose a path from $\Vec v$ to $z \in \Delta$, then via the homomorphism $\rho : \pi_1(\A^1\setminus\Delta,\Vec v)\to\mathfrak S_d$ we obtain a well-defined monodromy around $z$, which is an element of $\mathfrak S_d$. In general the monodromy depends nontrivially on the choice of path, but its conjugacy class in $\mathfrak S_d$ is well-defined. In particular, we may write $N_\rho(z)$ for the value of $N(-)$ on the monodromy around $z$. We may also consider the monodromy of $\rho$ around $\infty$, which is the monodromy along the boundary of a disk containing $\Delta$; it is well-defined and independent of choices. 

Geometric points of $M_g^\fr(\P^1,\lambda)_\pi$ are then in bijection with pairs $(\Delta,\rho)$, where $\Delta \subset \A^1$ is a finite set, and $\rho : \pi_1(\A^1 \setminus \Delta, \Vec{v}) \to \mathfrak S_d$ satisfies:
\begin{enumerate}[(i)]
    \item The monodromy of $\rho$ around $\infty$ equals  $\pi$.
    \item $N_\rho(z)>0$ for all $z \in \Delta$.
    \item $2g-2 = -2d + \sum_{z \in \Delta \cup \{\infty\}} N_\rho(z)$.
    \item The image of $\rho$ is a transitive subgroup of $\mathfrak S_d$.
\end{enumerate}
Condition (iv) ensures that the covering curve is connected and hence condition (iii) ensures that it has genus $g$.

Finally, let us describe the analytic topology on the complex points of $M_g^\fr(\P^1,\lambda)_\pi$ in these terms. A basis for the topology is indexed by pairs $(\mathfrak U, \sigma)$, where $\mathfrak U \subset \C$ is an open subset with compact closure, and $\sigma : \pi_1(\C \setminus \mathfrak U,\Vec v) \to \mathfrak S_d$ is a homomorphism. The open set corresponding to $(\mathfrak U,\sigma)$ is
\[ B(\mathfrak U,\sigma) := \{(\Delta,\rho) : \Delta \subset \mathfrak U \text{ and } \sigma = \rho \circ i\}.\]
Here $i : \pi_1(\C \setminus \mathfrak U,\Vec v) \to \pi_1(\C \setminus \Delta,\Vec v)$
is the natural homomorphism. Informally, the space is topologized in such a way that moving around in the moduli space corresponds to perturbing the branch locus; moreover, branch points are allowed to collide whenever the monodromy allows it. 


\subsection{A consequence of the Riemann--Roch theorem}
Let $M_{g,1}$ be the moduli stack parametrizing pairs $(C,x)$ of smooth projective curves of genus $g$ with a marked point $x$, and $M_g^1$ the space parametrizing the additional datum of a nonzero tangent vector at the distinguished marked point. The space $M_g^1$ is a scheme, with the sole exception of $M_0^1 \cong B\mathbb G_a$. One has $M_{g,1} = [M_g^1/\mathbb G_m]$.

For all $d \geq 0$, let $V_d$ be the bundle over $M_{g,1}$ whose fiber over $(C,x)$ is the datum of a section of  $\mathcal O_C(d\cdot x)$, i.e.\ a meromorphic function on $C$, with a pole of order at most $d$ at $x$ and holomorphic away from $x$. According to the Riemann--Roch theorem, the projection $V_d \to M_{g,1}$ is a vector bundle of rank $d-g+1$, whenever $d>2g-2$. We have a chain of subbundles: 
$$\ldots \subset V_d \subset V_{d+1} \subset V_{d+2} \subset \ldots$$ 
and there is an isomorphism $V_d \setminus V_{d-1} \cong M_g(\P^1,(d))$, where $(d)$ denotes the partition of $d$ into a single block; indeed, a meromorphic function with a unique pole at $x$ of exact order $d$ is the same as a degree $d$ branched cover of $\P^1$ with total ramification over $\infty$ at $x$.

Let $V_d'$ be the pullback of the bundle $V_d$ to $M_g^1$. As before we denote by $\alpha$ the permutation $(12\ldots d)$. There is a natural embedding
$$ \phi : M_g^\fr(\P^1,(d))_\alpha \hookrightarrow V_d',$$
since the projection $M_g(\P^1,(d))\to M_{g,1}$ lifts to $M_g^1$ on the cover $M_g^\fr(\P^1,(d))_\alpha.$ We now claim that the image of $\phi$ is a translate of the subbundle $V_{d-1}'$. This amounts to the following assertion: suppose sections $\sigma,\tau$ of $\mathcal O_C(d\cdot x)$ are given, and that $\sigma$ corresponds to a branched cover $C \to \P^1$ of degree $d$ with $x$ as a point of total ramification over $\infty$, taking a distinguished tangent vector at $P$ to the distinguished tangent vector at $\infty$ of $\P^1$. Then $\tau$ has the same property if and only if $\sigma - \tau$ is a section of $\mathcal O_C((d-1)\cdot x).$ Indeed, this claim is local, and can be verified in a complex-analytic neighborhood of $x$ on $C$. We may therefore consider the case $C=\P^1$, $x=\infty$, in which case a branched cover preserves a distinguished tangent vector at infinity precisely when it is defined by a \emph{monic} polynomial, and we arrive at the familiar statement that if $f$ and $g$ are degree $d$ polynomials with $f$ monic, then $g$ is monic if and only if $\deg(f-g)<d$. The above discussion and the Riemann--Roch theorem thus implies the following theorem:

\begin{thm}\label{RR}Suppose that $d>2g-1$. Then the projection $M_g^\fr(\P^1,(d))_\alpha \to M_g^1$ is an affine bundle of rank $d-g$.
\end{thm}

In particular, to prove the Mumford conjecture it suffices to compute the rational cohomology of $M_g^\fr(\P^1,(d))_\alpha$ in a stable range. In the rest of this paper we will carry  this out, using more topological than algebro-geometric methods. From this point on, we will no longer distinguish between schemes, varieties, etc., and their sets of complex points. 


\section{Topological branched covers}\label{sec:2}

By a \emph{surface}, we mean a topological two-dimensional oriented manifold, possibly non-compact, possibly with boundary. To avoid trivialities, all surfaces are moreover assumed to be connected. 

The algebro-geometric notion of branched cover extends in the evident way to the purely topological setting. Namely, a continuous map $f: \Sigma' \to \Sigma$ between surfaces is said to be a \emph{branched cover} if it is proper, and a local homeomorphism outside of finitely many points in the interior of $\Sigma'$, called the \emph{ramification points}. At each of these finitely many points, $f$ can be written in local coordinates as the map $\C \to \C$ given by $z \mapsto z^r$, where $r>1$ is called the \emph{ramification index}. The images of ramification points are called \emph{branch points}. Every finite-sheeted covering space of the complement of a finite set in the interior of a surface extends uniquely to a branched cover. 

		Let $A \subset \Sigma$. Let $f: \Sigma' \to \Sigma$ be a branched cover of degree $d$. A \emph{trivialization} of $f$ over $A$ is an isomorphism 	over  $A$:    \[\begin{tikzcd}
			f^{-1}( A) \arrow{rd}{f}\arrow{rr}{\cong} & & \coprod_d A \arrow{ld}\\
			&  A & 
		\end{tikzcd}\]
		Note that the existence of a trivialization over $A$ implies that $A$ is disjoint from the branch locus of $f$.

We will in this section define a space $\Bra^d_{\Sigma\rel A}$ parametrizing degree $d$ branched covers of $\Sigma$ with a trivialization along $A$. It will be topologized in such a way that branch points can move around and collide with each other, if the monodromy allows it. Branch points can, in addition, vanish by wandering off to infinity along the directions where $\Sigma$ is not compact. (We strongly suggest that the reader compare this with the spaces $\mathsf{Fin}$ and $\mathsf{Hur}$ defined in \cite[Section~5]{BDPW}.) We will also see that if $\mathbb D$ denotes the closed disk and $y \in \partial \mathbb D$ is a boundary point, then $M_g^\fr(\P^1,\lambda)_\pi$ is homeomorphic to a connected component of $\Bra^d_{\mathbb D \rel y}$, for all $g$, and every $\pi \in \mathfrak S_d$ with cycle-type $\lambda$. 

If $A \neq \varnothing$ then  $\Bra^d_{\Sigma\rel A}$  is a topological space, but if $A = \varnothing$ then a branched cover may have nontrivial automorphisms, and we are obliged to consider the space $\Bra^d_{\Sigma\rel A}$ as a \emph{topological stack}. Helpful references for topological stacks are \cite{pardon1,pardon2}. Although it will be important for us that we work with stacks, the particular implementation details are not so critical. For us, a topological stack $X$ will be a sheaf of groupoids on the site of topological spaces with respect to the open cover topology, with proper diagonal, which admits a representable surjection $U \to X$ from a topological space $U$ such that $U \to X$ admits local sections. We could equally well have considered sheaves of groupoids on the sites of smooth manifolds (differentiable stacks), or on the site of compact Hausdorff spaces with respect to the topology of finite closed covers (condensed groupoids). On the other hand, it would not be sufficient to merely work with orbifolds; the spaces of branched covers that we consider will have singularities and be infinite-dimensional, as soon as the base surface fails to be compact.

\subsection{Definition of \texorpdfstring{$\Bra^d_{\Sigma\rel A}$}{Bra} when \texorpdfstring{$A$}{A} is nonempty.}\label{sec:def of Bra}

Assume that $A\neq \varnothing$. Let $\Bra^d_{\Sigma\rel A}$ be the set\footnote{More properly, we should say ``isomorphism classes of branched covers'', but if two branched covers are isomorphic compatibly with their trivializations over $A$ then such an isomorphism is unique.} of all degree $d$ branched covers of $\Sigma$, equipped with a trivialization along $A$. We topologize $\Bra^d_{\Sigma\rel A}$ as follows. A basis is indexed by triples $(K,\mathfrak U,E)$ with $K \subset \Sigma$ a compact subset containing $A$, $\mathfrak U = U_1 \cup \ldots \cup U_k$ a finite union of open disks with disjoint closures inside $K \setminus A$, and $E \to K \setminus \mathfrak U$ a degree $d$ covering space equipped with a trivialization over $A$. For $f \in \Bra^d_{\Sigma\rel A}$, we write $f^{-1}(K \setminus \mathfrak U) \cong E$ to mean that $f^{-1}(K \setminus \mathfrak U)$ and $E$ are isomorphic as covering spaces of $K \setminus \mathfrak U$ with trivializations over $A$. Note that such an isomorphism is unique if it exists. 
 The basic open set corresponding to the triple $(K,\mathfrak U,E)$ is 
		$$ B(K,\mathfrak U,E) = \{ \text{covers } f: \Sigma' \to \Sigma \text{ such that } f^{-1}(K \setminus \mathfrak U) \cong E \text{ and } f^{-1}(\mathfrak U) \text{ is a disjoint union of disks}\}.$$
    Note that if $\Sigma$ is a compact surface, then we may always take $K=\Sigma$ when considering the basis for the topology on $\Bra^d_{\Sigma \rel A}$. On an open set $B(\Sigma,\mathfrak U,E)$ it is clear that each of the corresponding branched covers of $\Sigma$ have homeomorphic total space, so topological invariants of the covering surface define locally constant functions on $\Bra^d_{\Sigma \rel A}$. If $\Sigma$ is noncompact then this is far from the case. 

\subsection{Branched covers of the closed disk.}

Let $\mathbb D$ be the closed two-dimensional disk, and $y \in \partial \mathbb{D}$ a fixed point. We may think of $\mathbb D$ as the real-oriented blow-up of $\P^1$ at $\infty$, so that the interior of $\mathbb D$ is identified with $\mathbb A^1$, and $y$ is identified with a unit tangent vector at $\infty$ of $\P^1$. Since a branched cover can be described equivalently as a finite-sheeted covering space of the complement of a finite subset, we see that points of $\Bra^d_{\mathbb D\rel y}$ can be bijectively identified with pairs $(\Delta,\rho)$ where $\Delta \subset \mathbb A^1$ is a finite set, and $\rho : \pi_1(\mathbb A^1 \setminus \Delta,\Vec v) \to \mathfrak S_d$ is a homomorphism satisfying $N_\rho(z)>0$ for all $z \in \Delta$. Compare with the discussion in \S\ref{sec: space of branched covers of the line}.

In particular, there is an evident map of sets $M_g^\fr(\P^1, \lambda)_{\pi} \to \Bra^d_{\mathbb D \rel y}$, for every $\pi \in \mathfrak S_d$ with conjugacy class $\lambda$, and all $g \geq 0$. 

\begin{prop}\label{prop:algebraic vs topological branched covers}
    The map $M_g^\fr(\P^1,\lambda)_\pi \to \Bra^d_{\mathbb D \rel y}$ is a homeomorphism onto a connected component. 
\end{prop}

\begin{proof}As discussed in \S\ref{sec:def of Bra}, compactness of $\mathbb D$ implies that connected genus $g$ covers of $\mathbb D$ form an open and closed subspace of $\Bra^d_{\mathbb D \rel y}$. The monodromy along the boundary is similarly constant on each open set $B(\mathbb D,\mathfrak U,E)$. 

Thus we just need to compare the topology defined in \S\ref{sec:def of Bra} with the one described in \S\ref{sec: space of branched covers of the line}. As it stands, they are almost the same. It remains to note that in \S\ref{sec: space of branched covers of the line} it suffices to consider open sets $\mathfrak U$ given by finite families of open disks with disjoint closures; moreover, shrinking $\mathfrak U$ further if necessary we may even assume that the cover over $\mathfrak U$ is a union of disks. 
\end{proof}
\begin{rem}The reader may ask why the condition that $f^{-1}(\mathfrak U)$ is a union of disks, which appears in the definition of the basic open sets of \S\ref{sec:def of Bra}, does not seem to have an analogue in the description of the topology in \S\ref{sec: space of branched covers of the line}. This condition appears in \S\ref{sec:def of Bra} in order to rule out, for example, two additional branch points with opposite monodromy appearing inside one of the disks $U_i$; such a cover should not be thought of as ``close'' to the ones in $B(K,\mathfrak U,E)$. The point is that such additional branch points will raise the genus of the covering surface, whereas in \S\ref{sec: space of branched covers of the line} the topology of the covering surface is fixed (it is a connected genus $g$ surface), making the condition automatic (possibly after shrinking $\mathfrak U$).\end{rem}

\begin{rem}\label{rem: connected components} From \cref{prop:algebraic vs topological branched covers}, and the connectedness of $M_g^\fr(\P^1,\lambda)_\pi$, it is not hard to deduce a description of the connected components of $\smash{\Bra_{\mathbb D \rel y}^d}$. Namely, a connected component is uniquely specified by a triple $(\pi,F,g)$, where $\pi \in \mathfrak S_d$ (recording the monodromy along $\partial \mathbb D$), $F$ is a set-partition of $\{1,\ldots,d\}$  whose blocks are unions of cycles of $\pi$ (recording which sheets above $y$ lie in the same connected component of the branched cover), and $g : F \to \Z_{\geq 0}$ is a function (recording the genera of the respective connected components of the branched cover). The function $g$ must have the property that it maps singleton blocks to $0$, since a degree $1$ cover of the disk cannot have positive genus.
\end{rem}

\subsection{Gluing of stacks}Gluing of schemes is described in many places in the literature, see e.g.\ \cite[Exercise~II.2.XX]{hartshorne}. The analogous description of how to glue stacks does not seem to appear in any standard reference, so we provide it here for the reader's convenience. We work in the setting of topological stacks, but the description holds in the algebraic setting, too. 

    Before giving the description, we make the remark that although stacks form a $(2,1)$-category, the family of all \emph{substacks} of a given stack is a \emph{set} (or more properly, a $0$-category). This means that two substacks of a given stack can only be isomorphic if they coincide; there is no additional choice of an isomorphism. 

Let $\{U_i\}_{i\in I}$ be a family of stacks on the site of topological spaces. Suppose we are given for all $i,j \in I$ an open substack $U_{ij} \hookrightarrow U_i$, with $U_{ii}=U_i$. Write $U_{ijk} = U_{ij} \cap U_{ik}$, and $U_{ijkl} = U_{ij} \cap U_{ik} \cap U_{il}$. 

Suppose we are further given:
\begin{enumerate}
	\item for each $i,j$ an isomorphism $\phi_{ij} : U_{ij} \to U_{ji}$, which maps $U_{ijk}$ isomorphically onto $U_{jik}$ for all $k$, with $\phi_{ji}=\phi_{ij}^{-1}$. 
	\item for each $i,j,k$, a $2$-morphism $\alpha_{ijk}$ as in the following diagram:
	\[\begin{tikzcd}
		U_{ijk}=U_{ikj}\arrow{rr}{\phi_{ik}} \arrow[dr,"\phi_{ij}"', start anchor = south]& \arrow[Rightarrow,d,"\alpha_{ijk}" description] & U_{kij}=U_{kji} \\
		& U_{jik} = U_{jki} \arrow[ur,"\phi_{jk}"', end anchor = south]& 
	\end{tikzcd}\]
\end{enumerate}We demand that there is a  commuting tetrahedron 
	\[\begin{tikzcd}
		& U_{ijkl} \arrow[ld,"\phi_{ij}"'] \arrow[dd,"\phi_{il}" near end] \arrow[rd,"\phi_{ik}"] & \\
		U_{jkli} \arrow[rd,"\phi_{jl}"'] \arrow[rr,"\phi_{jk}" near start, crossing over]& & U_{klij} \arrow[ld,"\phi_{kl}"]\\
		& U_{lijk} & 
	\end{tikzcd}\]
whose four faces, not indicated in the diagram, are given by the $2$-morphisms $\alpha_{ijk}$, $\alpha_{ijl}$, $\alpha_{ikl}$, and $\alpha_{jkl}$.

Then we may glue the family $\{U_i\}_{i\in I}$ to a unique stack $U$. Explicitly, $U$ represents the following functor: a map $S \to U$ (that is, an object of $U(S)$) is an open cover $S= \cup_{i\in I} S_i$, maps $f_i : S_i \to U_i$ such that $f_i^{-1}(U_{ij})=S_i \cap S_j$ for all $i,j$, and for each $i,j$ an isomorphism between the restrictions of $f_i$ and $f_j$ to $S_i \cap S_j$ (where we identify $U_{ij}(S_i \cap S_j)$ with $U_{ji}(S_i \cap S_j)$ via $\phi_{ij}$), satisfying the usual cocycle conditions. A morphism in $U(S)$ from $\{(S_i,f_i)\}$ to $\{(S_i',f_i')\}$ is a family of isomorphisms $f_i|_{S_i \cap S_j'} \stackrel\sim\to f_j'\vert_{S_j' \cap S_i}$ which agree on overlaps.  This $U$ is, first of all, a stack. Moreover, it is a topological stack if each $U_i$ is a topological stack: if $V_i\to U_i$ is an atlas for each $i$, then $\coprod_i V_i \to U$ is an atlas for $U$. 

Let us also consider the special case of gluing together a family of \emph{quotient stacks}. Suppose in the above situation that each stack $U_i$ is a quotient stack $[X_i/G_i]$, and for all $i,j$ we are given an open $G_i$-invariant subspace $X_{ij} \subseteq X_i$. We let $X_{ijk}=X_{ij}\cap X_{ik}.$ The gluing data may then be presented somewhat more concretely as choosing a $G_i$-equivariant $G_j$-torsor $Y_{ij}\to X_{ij}$ for each $i, j$, and a $(G_i \times G_j)$-equivariant isomorphism $\phi_{ij} : Y_{ij}\to Y_{ji}$ such that $\phi_{ij}=\phi_{ji}^{-1}$, with the property that $\phi_{ij}$ carries the inverse image of $X_{ijk}$ to the inverse image of $X_{jik}$. The $2$-morphism $\alpha_{ijk}$ corresponds to the choice of an isomorphism between two natural $G_i$-equivariant $G_k$-torsors over $X_{ijk}$: we may take the restriction of $Y_{ik} \to X_{ik}$, or we may take the $G_j$-equivariant $G_k$-torsor 
$$ Y_{ji} \times_{X_j} Y_{jk} \to Y_{ji} \times_{X_j} X_{jk},$$
pull it back along $\phi_{ij}$, and use $G_j$-equivariance to descend it to $X_{ijk}$.

\subsection{Definition of \texorpdfstring{$\Bra^d_{\Sigma\rel A}$}{Bra} when \texorpdfstring{$A$}{A} is empty.}

We are now ready to define $\Bra^d_{\Sigma\rel A}$ when $A=\varnothing$. In this case we omit ``$\mathsf{rel}\, A$'' from the notation. 

If $\Sigma$ has boundary, it is quite easy to define $\Bra_\Sigma^d$. Choose a point $y \in \partial \Sigma$, and make the definition
$$ \Bra_\Sigma^d := [\Bra^d_{\Sigma\rel y} / \mathfrak S_d]$$
 where $\mathfrak S_d$ acts by changing the trivialization at $y$, and $[-/\mathfrak S_d]$ denotes the stack quotient. 

When $\partial \Sigma=\varnothing$, the above definition cannot just be copied, since there is no point at which we are guaranteed that a trivialization should exist: every point of $\Sigma$ is potentially a branch point. If we choose an interior point $x\in \Sigma$ then we can still form the quotient stack $[\Bra^d_{\Sigma\rel x} / \mathfrak S_d]$, but it will only give us the open substack of $\Bra^d_\Sigma$ defined by the condition that there is no branch point at $x$. However, if we consider the family of all stacks $[\Bra^d_{\Sigma\rel x} / \mathfrak S_d]$ as $x$ varies, then these stacks are going to form an open cover of $\Bra_\Sigma^d$, and we may define $\Bra_\Sigma^d$ by gluing. Since we are gluing a family of quotient stacks, we are in precisely the setting of the final paragraph of the previous subsection. 

To ease notation, set $X_x = \Bra^d_{\Sigma \rel x}$, and $U_x = [X_x/\fS_d]$. For $x, y \in \Sigma$, let $X_{xy}$ be the open subspace of $X_x$ defined by the condition of having no branch point at $y$. For $x\neq y \in \Sigma $ let $Y_{xy}=\Bra_{\Sigma \rel \{x,y\}}$. We let also $Y_{xx}$ be the space parametrizing branched covers of $\Sigma$ equipped with \emph{two} trivializations at $x$. Then $Y_{xy}$ is an $\fS_d$-equivariant $\fS_d$-torsor over $X_{xy}$ for all $x,y$. The identity map is an $(\fS_d \times \fS_d)$-equivariant isomorphism $\phi_{xy} : Y_{xy}\to Y_{yx}$. The isomorphism on triple overlaps $X_{xyz}$ is the obvious one coming from the fact that both $\fS_d$-torsors over $X_{xy} \cap X_{xz}$ are identified with the open subspace of $\Bra_{D \rel \{x,z\}}^d$ where there is no branching at $y$, and one may similarly verify the tetrahedron equation. This is exactly the gluing data needed to glue together the quotient stacks $U_x$, and thereby to define the space $\Bra_\Sigma^d$ in general.





\subsection{Functoriality}\label{functorial for open immersions} 
A certain amount of functoriality of the construction $\Bra$ will be important in our arguments. Specifically, if $i:U \hookrightarrow \Sigma$ is an open embedding, then there is a continuous map $i^\ast:\Bra^d_{\Sigma \rel A} \to \Bra^d_{U \rel (U\cap A)}$, informally defined by taking a branched cover of $\Sigma$ and restricting it to $U$. When $U\cap A \neq \varnothing$ this is a direct verification in terms of the topology defined in \cref{sec:def of Bra}. When $U \cap A = \varnothing$ the map is instead obtained by gluing together the maps $[\Bra^d_{\Sigma \rel (A\cup\{x\})}/\fS_d] \to [\Bra^d_{U \rel  x}/\fS_d]$ as $x$ ranges over $U$.

 \subsection{The moduli problem} To a reader accustomed to stacks in algebraic geometry, our usage of stacks in this section may seem a bit artificial. It would be more natural to describe $\Bra_\Sigma^d$ as a stack by exhibiting a geometrically meaningful functor which it represents. When $\Sigma$ is the closed disk, the topological stack $\Bra_{\Sigma \rel y}^d$ can be obtained as the complex points of an algebraic stack which represents a meaningful functor, according to \cref{prop:algebraic vs topological branched covers}. This tautologically gives rise to a functor that $\Bra_{\Sigma \rel y}^d$ itself represents \cite[Section~4]{mikala}, but this is not a very convenient description. 

 We do have a candidate for a functor represented by $\Bra_\Sigma^d$, as follows. 	Let $S$ and $T$ be topological spaces.  Let $f : T \to S \times \Sigma$ be a continuous map. Write $T^\circ = f^{-1}(S \times \Sigma^\circ)$, and $\partial T = f^{-1}(S \times \partial \Sigma)$. We say that $f$ is a \emph{family of branched covers of $\Sigma$ parametrized by $S$} if:
		\begin{enumerate}
			\item $f$ is universally closed. 
			\item $T \to S$ is locally on $T^\circ $ isomorphic to the projection $S \times \R^2 \to S$.
			\item If $U \subset T$ is the maximal open subset on which $f$ is a local homeomorphism, then $U$ contains $\partial T$, and $(T \setminus U) \to S$ has finite fibers.
		\end{enumerate}
  It seems likely that $\Bra^d_\Sigma$ represents the functor assigning to a topological space $S$ the groupoid of all families of degree $d$ branched covers of $\Sigma$ parametrized by $S$, possibly under some point-set hypotheses on $S$. However, it will not be necessary for our arguments to know this. 

\section{Stabilization}\label{sec:3}

For a surface $\Sigma$, there is a natural \emph{stabilization map} $\Bra_\Sigma^d \hookrightarrow \Bra_\Sigma^{d+1}$, defined by adding a trivial sheet to get a disconnected cover. If $\Sigma$ is compact, then the above map may be identified with the inclusion of a union of connected components. We set \[\Bra_{\Sigma} := \varinjlim_d \Bra_{\Sigma}^d.\]
Now let $\square = [0,1]^2$, $\squareopen =  [0,1] \times\{0\} \cup \{0,1\} \times [0,1] $, $\squaretop = [0,1] \times \{1\}$, $\partial = \squaretop \cup \squareopen$. The space $\Bra_{\square \rel \squareopen}^d$ is a topological $E_1$-algebra for all $d$, as illustrated schematically in the following figure:
\[\begin{tikzpicture}[baseline={(0,0.4)}]
\draw (1,1) -- (0,1);
\draw[dashed] (0,1) -- (0,0) -- (1,0) -- (1,1);
\end{tikzpicture} \qquad \begin{tikzpicture}[baseline={(0,0.4)}]
\draw (1,1) -- (0,1);
\draw[dashed] (0,1) -- (0,0) -- (1,0) -- (1,1);
\end{tikzpicture} \qquad \leadsto \qquad \begin{tikzpicture}[baseline={(0,0.4)}]
\draw (2,1) -- (0,1);
\draw[dashed] (0,1) -- (0,0) -- (1,0) -- (1,1);
\draw[dashed] (1,0) -- (2,0) -- (2,1);
\end{tikzpicture} \]
The inclusion $\Bra_{\square \rel \squareopen}^d \hookrightarrow \Bra_{\square \rel \squareopen}^{d+1}$ is a morphism of $E_1$-algebras, and the colimit $\Bra_{\square \rel \squareopen}$ is again an $E_1$-algebra.

We say that a connected component of $\Bra_{\square \rel \squareopen}^d$ is a \emph{good component of genus $g$} if it parametrizes genus $g$ branched covers of the square such that the monodromy along $\partial$ is a $d$-cycle, and $d>2g-1$. Every good component of genus $g$ is homotopy equivalent to $M_g^1$, by \cref{RR} and \cref{prop:algebraic vs topological branched covers}. We define the good components of $\Bra_{\square \rel \squareopen}$ to be the union of the good components of $\Bra_{\square \rel \squareopen}^d$ for all $d$.

\begin{prop}\label{prop: good components}
    Let $X$ and $Y$ be good components of $\Bra_{\square \rel \squareopen}$, with $X \simeq M_g^1$ and $Y \simeq M_h^1$, respectively. Suppose $a \in \Bra_{\square \rel \squareopen}$ is such that $a \cdot X \subseteq Y$, i.e.\ multiplication by $a$ maps $X$ into $Y$. Then $h \geq g$, and $a\cdot -$ is homotopic to the Harer stabilization map. 
\end{prop}

\begin{proof}By \cref{prop:algebraic vs topological branched covers}, we have $X \cong M_g^\fr(\P^1,\lambda)_\pi$ for some  $\pi$. In particular, $X$ comes with a universal family of branched covers. The total space is a surface bundle over $X$ with a section, and a marked tangent vector along the section. Taking the real oriented blow-up along the section we get a surface bundle with boundary, and hence a natural map from $X$ to $[\ast/\mathrm{Homeo}_\partial^+(\Sigma_g)]$, the topological moduli stack of topological genus $g$ surface bundles with trivialized boundary, which is a weak equivalence. Similarly we get $Y \to [\ast/\mathrm{Homeo}_\partial^+ (\Sigma_h)]$. The diagram
    \[\begin{tikzcd}
        X \arrow{d}{a \cdot -} \arrow[r] & \hspace{0pt} [\ast/\mathrm{Homeo}_\partial^+(\Sigma_g)] \arrow[d]\\
        Y \arrow[r] & \hspace{0pt} [\ast/\mathrm{Homeo}_\partial^+(\Sigma_h)]
    \end{tikzcd}\]
    commutes, where the right vertical map is given by boundary connect sum with a fixed topological surface corresponding to  $a$. But boundary connect sum with a fixed surface is Harer stabilization. 
\end{proof}



With this proposition in place, it is now natural to apply the group-completion theorem (\cref{groupcompletion}) to $\Bra_{\square \rel \squareopen}$. We now verify that $\Bra_{\square \rel \squareopen}$ satisfies the hypotheses of the group-completion theorem, starting with the Ore conditions. Let $\varpi : \Bra_{\square \rel \squareopen}\to \mathfrak S_\infty$ be the monodromy homomorphism, recording the monodromy of the branched cover along $\partial$. The monodromy map is equivariant, where $\mathfrak S_\infty$ acts on itself by conjugation, and on $\Bra_{\square \rel \squareopen}$ by changing the trivialization along $\squareopen$. We denote the latter action symbolically by $(\pi, a) \mapsto a^\pi$, where $\pi \in \mathfrak S_\infty$ and $a \in \Bra_{\square \rel \squareopen}$. If we think of points of $\Bra_{\square \rel \squareopen}$ as pairs $(\Delta,\rho)$ with $\Delta$ a finite subset of the interior of $\square$, and $\rho : \pi_1(\square\setminus \Delta) \to \mathfrak S_\infty$ (with a basepoint chosen along $\squareopen$), then the action of $\mathfrak S_\infty$ is simply by conjugating such homomorphisms. 

\begin{lem}
    The following two maps $\Bra_{\square \rel \squareopen} \times \Bra_{\square \rel \squareopen} \to \Bra_{\square \rel \squareopen}$ are homotopic:
    \begin{align*}
        (a,b) & \mapsto a \cdot b \\
        (a,b) & \mapsto b \cdot a^{\varpi(b)}.
    \end{align*}
    In particular, the two corresponding maps $H_\bullet(\Bra_{\square \rel \squareopen};\Z) \times \pi_0(\Bra_{\square \rel \squareopen}) \to H_\bullet(\Bra_{\square \rel \squareopen};\Z)$ coincide, which implies that $\pi_0(\Bra_{\square \rel \squareopen}) \subset H_\bullet(\Bra_{\square \rel \squareopen};\Z)$ satisfies the right (and left) Ore conditions. 
\end{lem}

\begin{proof}Rotate the configurations of branch points of $a$ and $b$ past each other, as in the figure below:
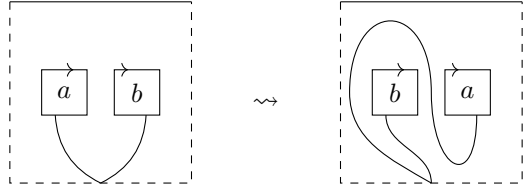
\[\begin{tikzpicture}[baseline={(0,1)},scale=0.6]
\draw (0,0) to [bend right] (1,1.5);
\draw[->-] (-1,1.5) to (-1.3,1.5) to (-1.3,2.5) to (-0.3,2.5) to (-0.3,1.5) to (-1,1.5);
\node at (0.8,2) {$b$};
\draw (0,0) to [bend left] (-1,1.5);
\draw[->-] (1,1.5) to (0.3,1.5) to (0.3,2.5) to (1.3,2.5) to (1.3,1.5) to (1,1.5);
\node at (-0.8,2) {$a$};
\draw[dashed] (-2,4) -- (-2,0) -- (2,0) -- (2,4);
\draw (-2,4) -- (2,4);
\end{tikzpicture} \qquad \leadsto \qquad \begin{tikzpicture}[baseline={(0,1)},scale=0.6]
\draw (0,0) to [out=90,in=270] (-1,1.5);
\draw[->-] (-1,1.5) to (-1.3,1.5) to (-1.3,2.5) to (-0.3,2.5) to (-0.3,1.5) to (-1,1.5);
\node at (0.8,2) {$a$};
\draw (0,0) to [out=150,in=270] (-1.8,2) to [out=90,in=90,looseness=3] (0,2) to[out=270,in=270, looseness=4] (1,1.5);
\draw[->-] (1,1.5) to (0.3,1.5) to (0.3,2.5) to (1.3,2.5) to (1.3,1.5) to (1,1.5);
\node at (-0.8,2) {$b$};
\draw[dashed] (-2,4) -- (-2,0) -- (2,0) -- (2,4);
\draw (-2,4) -- (2,4);
\end{tikzpicture}\]
Thinking of points of $\Bra_{\square \rel \squareopen}$ in terms of configurations and homomorphisms out of the fundamental group, the figure shows that the effect on monodromy is as claimed. 
\end{proof}

Recall from \cref{rem: connected components} the description of connected components of $\Bra_{\square \rel \squareopen}^d$ as  triples $(\pi,F,g)$, with $\pi \in \mathfrak S_d$, $F$ a set-partition of $\{1,\ldots,d\}$  whose blocks are unions of cycles of $\pi$, and $g : F \to \Z_{\geq 0}$ a function taking singletons to zero. In these terms, the monoid operation on $\pi_0 \smash{\Bra^d_{\square \rel \squareopen}}$ is given by 
\[ (\pi,F,g) \ast (\pi',F',g') = (\pi \circ \pi',F \vee F', g \ast g').\]
Here $\pi \circ \pi'$ denotes multiplication of permutations, $F \vee F'$ denotes taking the join in the partition lattice, and the formula for $g \ast g'$ is somewhat complicated: for $S \in F \vee F'$, one has
\[ (g \ast g')(S) - 1 =  |S| + \sum_{\substack{T \in F \\ T \subseteq S}} (g(T)-1) + \sum_{\substack{T' \in F' \\ T' \subseteq S}} (g'(T')-1).\]
This is most easily justified by a calculation with Euler characteristic, and the fact that the surfaces are connected. The monoid $\pi_0\Bra_{\square \rel \squareopen}$ is obtained from $\smash{\pi_0\Bra_{\square \rel \squareopen}^d }$ by taking the union as $d$ grows. 

\begin{lem}The category $E\pi_0 \Bra_{\square \rel \squareopen}$ is filtered. The subcategory of all good components and the subcategory $E\pi_0 \Bra_{\square \rel \partial}$ are also filtered, as well as cofinal.\end{lem}

\begin{proof}Let $S$ be a monoid. That $ES$ is filtered is equivalent to the following two conditions: for any $s,t\in S$ there are $u,v \in S$ such that $su=tv$; for any $r,s,t \in S$ with $rs=rt$, there is $u \in S$ such that $su=tu$. When this holds, a subset $A \subset S$ defines a filtered cofinal subcategory of $ES$ if for all $s \in S$ there is $u \in S$ such that $su\in A$.

Now take $S=\pi_0 \Bra_{\square \rel \squareopen}$. Let also $S^{(d)} = \pi_0 \Bra^d_{\square \rel \squareopen}$, so that $S$ is the increasing union of the submonoids $S^{(d)}$.

The first condition for $S$ to be filtered is clear from the previous lemma. For the second condition, we suppose that $rs=rt$ and choose $d$ so that $s,t \in S^{(d)}$. Choose any $u \in S^{(d)}$ for which the associated set-partition is the trivial partition with a single block (that is, a component parametrizing connected covers). Then $su=tu$. Indeed, the boundary monodromies must coincide since $rs=rt$ implies that $s$ and $t$ have the same boundary monodromy, and the connected components coincide since $u$ corresponds to a connected cover. That the genus of the two covers coincide follows by Euler characteristic considerations from the equation $rsu=rtu$. 
%
%
%

Let us now check the cofinality condition for the subcategory of good components. Pick $s \in S^{(d)}$. We can choose $v\in S^{(d)}$ so that the boundary monodromy of $sv$ is a $d$-cycle, since $\fS_d$ is a group. Let $g$ be the genus of the connected surface parametrized by $sv$. Suppose that $sv$ is not a good component, i.e. that $d <2g$. Let $w \in S^{(2g)}$ parametrize covers of $\square$ by disjoint unions of disks, with boundary monodromy $(d,d+1, \dots ,2g)$. Then $svw \in S^{(2g)}$ parametrizes degree $2g$ connected covers of $\square$ of genus $g$, with boundary monodromy a $2g$-cycle, so that $svw$ is a good component. The cofinality of $E\pi_0 \Bra_{\square \rel \partial}$ is simpler, using only that $\fS_\infty $ is a group.
\end{proof}

\begin{prop}$H_\bullet(\Omega_0 B \Bra_{\square \rel \partial};\Z) \cong \varinjlim_g H_\bullet(M_g^1;\Z).$\label{prop:Mg to delooping}
    
\end{prop}

\begin{proof}
    By the group-completion theorem, $H_\bullet(\Omega_0 B \Bra_{\square \rel \squareopen};\Z) \cong \varinjlim H_\bullet(X_s;\Z),$ with $X_s$ ranging over all connected components of $\Bra_{\square \rel \squareopen}$. By cofinality, we may instead take the colimit over good components, which by \cref{prop: good components} gives the colimit $\varinjlim_g H_\bullet(M_g^1;\Z)$, i.e.\ the stable homology of the mapping class group. On the other hand, we may, again by cofinality, also take the colimit over the components of the submonoid $\Bra_{\square \rel \partial}$ to get the same answer. But by another application of the group-completion theorem, this is $H_\bullet(\Omega_0 B \Bra_{\square \rel \partial};\Z)$. 
\end{proof}

\begin{prop}$H_\bullet(\Omega_0 B \Bra_{\square \rel \partial};\Z) \cong \varinjlim_d H_\bullet(\Omega_0 B \Bra^d_{\square \rel \partial};\Z)$.\label{prop:infinite d to finite d}
    
\end{prop}

\begin{proof}Recall that we defined $\Bra_{\square \rel \partial} := \varinjlim \Bra_{\square \rel \partial}^d$. This is in fact a homotopy colimit, i.e.\ a colimit in the $\infty$-category of spaces, since the maps $\Bra^d_{\square\rel\partial} \hookrightarrow \Bra^{d+1}_{\square\rel\partial}$ are obviously cofibrations. Moreover, it is also a colimit in the $\infty$-category of $E_1$-algebras, since the forgetful functor from $E_1$-algebras to spaces creates filtered colimits \cite[Proposition~3.2.3.1]{luriehigheralgebra}. 

    Furthermore, the group-completion functor $\Omega B$ is left adjoint to the inclusion of the $\infty$-category of group-like $E_1$-algebras into the $\infty$-category of $E_1$-algebras (see \cref{section about e1-algebras}), so it commutes with all colimits. In particular 
$\Omega B \Bra_{\square\rel\partial} \simeq \varinjlim \Omega B \Bra_{\square \rel \partial}^d$. But now taking homology commutes with filtered colimits. \end{proof}

\section{Delooping and scanning}\label{sec:4}

The arguments of this section follow \emph{very} closely those of \cite[Section~5]{BDPW}. Both have as goal to set up a scanning map to construct a delooping of a Hurwitz space. Both the arguments in \cite[Section~5]{BDPW} and in \cite{bianchi3} are in turn based on those of \cite{hatcher-madsenweiss}. 

Let $L$ and $M$ be $1$-manifolds, with $A \subset L$ and $B \subset M$ subsets. We introduce the notation
$$ \Bra^d_{(L \rel A) \times (M \rel B)} := \Bra^d_{(L \times M) \rel (A\times M \cup L \times B)}.$$
We omit `$\mathsf{rel}\, A$' if $A$ is empty, and similarly for $B$. Let $I=[0,1]$.
The space $\Bra^d_{(L \rel A) \times (I \rel \partial I)}$ is an $E_1$-algebra. The main theorem of this section gives a geometric model of a delooping of this $E_1$-algebra.

\begin{thm}\label{thm:scanning} Let $L$ be an open, half-open, or closed interval, with $A \subseteq \partial L$.
    Then $B \Bra^d_{(L\rel A) \times (I \rel \partial I)}$ is weakly homotopy equivalent to $ \Bra^d_{(L \rel A) \times \R}$. 
\end{thm}

Before giving the proof of \cref{thm:scanning} it will be convenient to set up some notation. For a topological monoid $M$, recall that the standard model for the bar construction $BM$ is as an explicit quotient of $\coprod_{p \geq 0} \Delta^p \times M^p$ \cite[\S5.2.4]{BDPW}. We write an element of the bar construction as a pair $((w_0,\ldots,w_p),(m_1,\ldots,m_p))$ with $w_i \in [0,1]$, $\sum w_i=1$, and $m_i \in M$.

We replace $\Bra^d_{(L \rel A) \times (I \rel \partial I)}$ with a weakly equivalent strictly associative $E_1$-algebra. Fix $\epsilon > 0$, and consider the space which parametrizes pairs of $s \in [0,\infty)$, and a point of $$\Bra^d_{(L \rel A) \times ([0,s] \rel [0,\epsilon) \cup (s-\epsilon,s])}.$$ This space is a unital topological monoid in an evident manner, and we denote this monoid $\mathfrak B$. There is a weak equivalence of $E_1$-algebras $\mathfrak B \simeq \Bra^d_{(L \rel A) \times (I \rel \partial I)}$, in much the same way as the space of Moore loops is homotopy equivalent to the usual loop space. 

Let $E$ denote the open interval $(-\epsilon,\epsilon)$. We can now define a map $\sigma : B \mathfrak B \to \Bra^d_{(L \rel A )\times E}$ by the following procedure: a point $x$ of $B\mathfrak B$ may be identified with a tuple $(w_0,\ldots,w_p) \in \Delta^p$ and a family of branched covers of $L \times [0, s_i]$, for $i=1,\ldots,p$, where $s_i \in [0,\infty)$. We multiply these $p$ branched covers in the monoid $\mathfrak B$; explicitly, this means that we set $t_i = s_1+\dots+s_{i}$, for $i=0,\dots,p$, translate the interval $[0,s_i]$ to $[t_{i-1},t_{i}]$, and glue together the $p$ branched covers to a branched cover of $L \times [t_0,t_p]$. After this, we let $t_\star = \sum_{i=0}^p w_i t_i$, and restrict the glued cover to a branched cover of the open subset $L \times (t_\star - \epsilon,t_\star + \epsilon) \cong L\times E$. The latter branched cover is $\sigma(x)$. The proof of \cref{thm:scanning} proceeds by showing that $\sigma$ is a weak equivalence.



\begin{proof}[Proof of \cref{thm:scanning}.]
    We consider first the case $\vert A\vert < 2$.

We first show surjectivity on homotopy groups. Let $q$ be a nonnegative integer, and choose a based map $f: S^q \to \Bra^d_{(L \rel A) \times E}$. We want to find a lift $g : S^q \to B\mathfrak B$ such that $f$ is homotopic to $\sigma \circ g$. 
	
	Choose (using compactness) a finite open cover $S^q = \bigcup_{\lambda \in \Lambda} U_\lambda$, such that there is a positive real $r>0$ and for each $\lambda$ an element $t_\lambda \in E$, with the following property:
	\emph{for any $x \in U_\lambda$, the branch locus of $f(x)$ is disjoint from $L \times (t_\lambda-r,t_\lambda+r)$.} Choose also a partition of unity $w_\lambda$ subordinate to the cover. 
	
	Consider $x \in U_{\lambda_0} \cap \ldots \cap U_{\lambda_k}$ (and assume it lies in no other $U_\lambda$). We assume that 
	$$ t_{\lambda_0} < \ldots < t_{\lambda_k}.$$
	Now take the corresponding branched covers of $L \times [t_{\lambda_0},t_{\lambda_1}]$, \ldots, $L \times [t_{\lambda_{k-1}},t_{\lambda_k}]$, respectively, and stretch them all in the $t$-coordinate by a factor $\tfrac{\epsilon}{r}$. Choosing trivializations of the branched covers along the endpoints of the intervals, we get a $k$-tuple of elements of our monoid $\mathfrak B$. (If $A = \varnothing$ such a trivialization can be chosen arbitrarily. If $\vert A \vert = 1$ there is a unique choice of trivialization which agrees with the fixed trivialization on $A$. If $\vert A \vert = 2$,  it may be that no compatible trivialization exists, which is why this case is treated separately.)  We define $g(x) \in B\mathfrak B$ to be the element defined by this $k$-tuple of elements of $\mathfrak B$, together with $(w_{\lambda_0}(x),w_{\lambda_1}(x),\ldots,w_{\lambda_k}(x)) \in \Delta^k$. 
	
	One checks that this is compatible with the identifications in the bar construction, and glues together well along overlaps for our open cover. To ensure that $g$ is a based map, we may insist that the basepoint of $S^q$ lies only in one of the sets $U_\lambda$ (if this is not the case, just delete the basepoint from all but one of the open sets in the cover).
	
	The branched cover given by $\sigma(g(x))$ is a ``zoomed in'' version of $f(x)$. More precisely, $\sigma(g(x))$ is the restriction of the branched cover $f(x)$ to the product of $L$ with a small interval of radius $r$, centered at the point $t_\ast(x) = \sum w_\lambda(x) t_\lambda$, which is then rescaled to be of radius $\epsilon$. In particular, $\sigma \circ g$ is homotopic to $f$ by continuously expanding the interval $(t_\ast(x)-r,t_\ast(x)+r)$ to fill up all of $E = (-\epsilon,\epsilon)$.
	
For injectivity on homotopy groups, consider instead a map $g : S^q \to B\mathfrak B$ and a nullhomotopy $f: \mathbb D^{q+1} \to \Bra^d_{(L \rel A) \times E}$ of $\sigma \circ g$. The same construction as above allows us to lift $f$ up to homotopy to a map $f' : \mathbb D^{q+1} \to B\mathfrak B$. We should construct a homotopy between $g$ and $f' \vert_{S^q}$. Consider an element $x \in S^q$, and let us suppose that $x \in U_{\lambda_1} \cap \ldots \cap U_{\lambda_k}$ (and $x$ lies in no other $U_\lambda$). Now $g(x)$ consists of a $p$-tuple of elements of $\mathfrak B$ given by branched covers of $L \times [t_0,t_1]$,\ldots, $L \times [t_{p-1},t_p]$, and a weight $w \in \Delta^p$. First, we continuously stretch all rectangles $L \times [t_{i-1},t_i]$ in the $t$-coordinate by a factor $\tfrac{\epsilon}{r}$. After stretching, they can be nontrivially written as a product in $\mathfrak B$, with the time-parameters $t_{\lambda_j}$, $j=1,\ldots,k$, providing the breakpoints where we can split up the rectangles. This gives us instead an equivalent point of the bar construction, given by $(p+k)$-tuple of elements of $\mathfrak B$ and a weight vector in $\Delta^{p+k}$, with $k$ zeroes inserted in the weight vector. Then continuously change the weights from the given weight vector $w$ to the one specified by the partition of unity $w_\lambda$, lowering the ``old'' weights to zero and raising the ``new'' weights from zero to their correct value $w_{\lambda_j}(x)$, $j=1,\ldots,k$. This finishes the proof if $\vert A \vert < 2$.

    Now assume that $L$ is a closed interval and $A = \partial L$. Write $A=A_- \cup A_+$ for the two endpoints of $L$. We have a fibration sequence of $E_1$-algebras
    $$ \Bra^d_{(L \rel A) \times( I \rel \partial I)} \to \Bra^d_{(L \rel A_-) \times (I \rel \partial I)} \to \mathfrak S_d$$
    where the second map records the monodromy along $A_+ \times I$. It deloops to a fibration sequence, fitting in a commutative diagram 
    \[\begin{tikzcd}
        B \Bra^d_{(L \rel A )\times (I \rel \partial I)} \arrow[r]\arrow[d]& B\arrow[r]\arrow[d] \Bra^d_{(L \rel A_-) \times (I \rel \partial I)} & \arrow{d}{\simeq} B \mathfrak S_d \\
        \Bra^d_{(L \rel A) \times E} \arrow{r}& \Bra^d_{(L \rel A_-) \times E}\arrow[r] & \hspace{0pt}[\ast/\mathfrak S_d]
    \end{tikzcd}\]
    where the bottom fibration sequence arises from the fact that $\Bra^d_{(L \rel A_-) \times E} \cong [\Bra^d_{(L \rel A) \times E}/\fS_d]$, where $\fS_d$ acts by changing the trivialization along $A_+ \times E$. 
    Since the middle vertical map is a weak equivalence by what we have already proved, so is the left vertical map.
\end{proof}



\begin{lem}\label{lem:connected}
    The space $\Bra^d_{(I \rel \partial I) \times \R}$ is connected.
\end{lem}

\begin{proof}
   We think of $\Bra^d_{(I \rel \partial I) \times \R}$ as parametrizing branched covers of the square, with a trivialization along the bottom and top sides, and where branch points can be created and annihilated along the left and right sides. Pick any point $x$ in this space. Continuously pushing all branch points off the right-hand side of the square defines a continuous path from $x$ to a trivial cover of the square. But a trivial cover of the square does not necessarily coincide with the basepoint of  $\Bra^d_{(I \rel \partial I) \times \R}$, since the trivializations of the trivial cover along the top and bottom sides may be incompatible; that is, there may be nontrivial monodromy along a line connecting the top and bottom sides. To define a continuous path to the basepoint of  $\Bra^d_{(I \rel \partial I) \times \R}$, we may further introduce a branch point along the left-hand side of the square, slide it horizontally across to the right-hand side, where it again vanishes. Under this process the monodromy along a line connecting the top and bottom sides changes according to the monodromy of a loop around the marking. Consequently we may in this manner connect any point of $\Bra^d_{(I \rel \partial I) \times \R}$ to the basepoint. 
\end{proof}

\begin{thm}\label{thm:delooping}
    $ B \Bra^d_{\square \rel \partial} = B \Bra^d_{(I \rel \partial I) \times (I \rel \partial I)} \simeq \Omega \Bra^d_{\R^2}$.
\end{thm}

\begin{proof}
We have $B \Bra^d_{(I \rel \partial I) \times (I \rel \partial I)} \simeq \Bra^d_{(I \rel \partial I) \times \R}$ by \cref{thm:scanning}. The right-hand side is again an $E_1$-algebra, and since it is moreover connected according to \cref{lem:connected} we have $\Bra^d_{(I \rel \partial I) \times \R} \simeq \Omega B \Bra^d_{(I \rel \partial I) \times \R} \simeq \Omega \Bra^d_{\R^2}$, where we used \cref{thm:scanning} again. The result follows. 
\end{proof}

\section{Local branched covers}\label{sec:5}

	\subsection{Preliminaries on CW-orbispaces}\label{subsec: CW}In this section, we complete the proof of the Mumford conjecture, by studying the topology of the delooping $\Bra^d_{\R^2}$ appearing in \cref{thm:delooping} by means of an explicit CW-decomposition. Since we are concerned with rational (co)homology the reader may in principle ignore stacky subtleties and work throughout with coarse spaces throughout this section of the paper; see \cref{rem:coarse-space-rational-equivalence}. However, the CW-decomposition will appear more naturally on the level of the stack. Let us first describe what a CW-decomposition even means in this context.
	
	\begin{defn}\label{CW def} A \emph{CW-orbispace} is a separated topological stack $X$, equipped with a family of representable maps $\{f_\alpha : [\mathbb D^{n_\alpha}/G_\alpha] \to X\}_{\alpha \in A}$, where for each $\alpha \in A$ we denote by $\mathbb D^{n_\alpha}$ the closed unit ball in $\mathbf R^{n_\alpha}$, and $G_\alpha \to \mathrm{O}(n_\alpha)$ is a linear representation of a finite group, satisfying the following axioms:
		\begin{enumerate}
			\item The restriction of $f_\alpha$ to the interior of $[\mathbb D^{n_\alpha}/G_\alpha]$ is an isomorphism onto its image. We denote the image by $e_\alpha$, and call $e_\alpha$ an \emph{open cell} of \emph{dimension} $n_\alpha$ in $X$. 
			\item The family of cells $\{e_\alpha\}_{\alpha \in A}$ are pairwise disjoint, and their set-theoretic union is all of $X$. 
			\item \emph{Closure-finite:} The image of $[\partial \mathbb D^{n_\alpha}/G_\alpha]$ under $f_\alpha$ is contained in a finite union of cells of dimension strictly less than $n_\alpha$. 
			\item \emph{Weak topology:} A substack $Z \subset X$ is closed if and only if $f_\alpha^{-1}(Z)$ is closed in $[\mathbb D^{n_\alpha}/G_\alpha]$ for all $\alpha$. 
		\end{enumerate}
		\end{defn}
	
	\begin{rem}
		If $G_\alpha$ is trivial for all $\alpha$, then $X$ must be a topological space, and the above definition reduces to the usual notion of a CW-complex. 
	\end{rem}

\begin{rem}
	Conditions (2) and (3) in the definition of a CW-orbispace can be checked on the level of the coarse space of $X$, since they are conditions on the underlying set. In fact (4) can be too, since closed substacks of $[\mathbb D^{n_\alpha}/G_\alpha]$ are $G_\alpha$-invariant closed subspaces of $\mathbb D^{n_\alpha}$, and a {$G_\alpha$-invariant} subspace of $\mathbb D^{n_\alpha}$ is closed if and only if its image in $\mathbb{D}^{n_\alpha}/G_\alpha$ is closed in the quotient topology. 
\end{rem}
	
	\begin{rem}
		If $X$ has finitely many cells, then axiom (4) can be omitted. 
	\end{rem}

\begin{rem}
	Our definition is nearly equivalent to the one used by Pardon \cite{pardon2}. Pardon requires instead all cells to be of the form $\mathbb D^{n_\alpha} \times [\ast/G_\alpha]$. This leads to the same class of spaces, since every closed cell  $[\mathbb D^{n_\alpha}/G_\alpha]$ itself admits a CW-decomposition in Pardon's sense. It will be very useful for us to allow for more general cells, with nonconstant isotropy groups, since it allows for much more efficient cell decompositions (in the sense of having fewer cells). 
\end{rem}

\begin{defn}A cell $e_\alpha$ in a CW-orbispace as in \cref{CW def} is \emph{orientable} if the map $G_\alpha \to \mathrm{O}(n_\alpha)$ lands in $\mathrm{SO}(n_\alpha)$.\end{defn} 

An increasing filtration of a topological stack $X$ by closed substacks $X^{(p)}$ gives rise to a spectral sequence
$$ E^1_{pq} = H_{p+q}(X^{(p)},X^{(p-1)};\Z) \implies H_{p+q}(X;\Z)$$
in the same manner as the spectral sequence associated to a filtered topological space. When $X$ is an ordinary CW-complex filtered by its skeleta, the resulting spectral sequence has its $E^1$-page concentrated along the row $q=0$, since $$H_\bullet(X^{(p)},X^{(p-1)};\Z) \cong \bigoplus_{\dim e_\alpha = p} H_\bullet^{\mathrm{BM}}(e_\alpha;\Z),$$
where $H_\bullet^{\mathrm{BM}}$ denotes Borel--Moore homology. The $q=0$ row coincides exactly with the usual \emph{cellular chain complex}, and in this way one may show that homology can be computed using cellular chains. 

Similarly, every CW-orbispace comes canonically with a filtration by skeleta. However, one can not as straightforwardly compute the homology from the cellular chains, as the Borel--Moore homology of the open cells $e_\alpha$ can be quite nontrivial --- for example, if $X$ has a single 0-cell $[\ast/G]$, then (co)homology of $X$ is group (co)homology of $G$. However, the situation simplifies if we take rational coefficients:\footnote{Or more generally, coefficients in an abelian group in which the orders of all isotropy groups are invertible.} in this case we have 
$$H_p^{\mathrm{BM}}(e_\alpha;\Q) \cong H_p(\mathbb D^{n_\alpha},\partial \mathbb D_{n_\alpha};\Q)_{G_\alpha} 
\cong \begin{cases}
    \Q & p=n_\alpha \text{ and } e_\alpha \text{ is orientable}, \\ 0 & \text{otherwise.}
\end{cases}$$ 
Indeed, $H_p(\mathbb D^{n_\alpha},\partial \mathbb D_{n_\alpha};\Q)$ vanishes for $p\neq n_\alpha$ and is spanned by the fundamental class $\mathfrak o$ when $p=n_\alpha$, and $G_\alpha$ preserves $\mathfrak o$ precisely if $e_\alpha$ is orientable. It follows that we may compute rational (co)homology of $X$ using an exact analogue of the usual cellular chain complex, defined instead by means of an orbispace CW-decomposition of $X$. Explicitly, define $C_p(X;\Q) = \bigoplus_{\dim e_\alpha=p} H_p^{\mathrm{BM}}(e_\alpha;\Q)$  --- or, less canonically, we may declare $C_p(X;\Q)$ to be the $\Q$-vector space with basis the set of $p$-dimensional orientable cells --- and define a differential $C_p(X;\Q) \to C_{p-1}(X;\Q)$ by taking the $E^1$-differential in the spectral sequence. Then the resulting chain complex computes the rational homology of $X$.

\begin{rem}\label{rem:coarse-space-rational-equivalence}
In fact, the rational (co)homology of $X$ coincides with the rational (co)homology of the coarse space $|X|$ of $X$. This can be seen from the Leray--Serre spectral sequence associated to the projection $X \to |X|$, using that all isotropy groups are finite. Moreover, the coarse space of a CW-orbispace is a CW-complex in the usual sense, and for a CW-decomposition in Pardon's sense (with constant isotropy along open cells) the sets of cells of the two decompositions are naturally identified.
\end{rem}

Let $g\colon X\to Y$ be a cellular map of CW-orbispaces, meaning that it carries the $k$-skeleton to the $k$-skeleton for all $k$. Then $g$ induces a map of spectral sequences associated to the respective filtrations, and hence a map $C_\bullet(X;\Q)\to C_\bullet(Y;\Q)$.\footnote{If $g$ is not cellular, one needs to replace it with a cellular approximation, i.e.\ a homotopic cellular map. It seems likely that there exists a version of the cellular approximation theorem for CW-orbispaces, but we do not know a reference. Fortunately, we will not need one, since in our case of interest we will construct a cellular approximation by hand.} Let us describe $C_\bullet(X;\Q)\to C_\bullet(Y;\Q)$ more explicitly. Consider a pair of $p$-cells $e_\alpha \subset X$ and $e_\beta \subset Y$. We have a zigzag 
$$ e_\alpha \leftarrow (e_\alpha \cap g^{-1}(e_\beta)) \to e_\beta $$
with the left map an open embedding and the right map proper. Borel--Moore homology is covariant for proper maps and contravariant for open embeddings, so this zigzag induces a map $H_\bullet^{\mathrm{BM}}(e_\alpha;\Q)\to H_\bullet^{\mathrm{BM}}(e_\beta;\Q)$. If we choose a generator of $H_p^{\mathrm{BM}}(e_\alpha;\Q)$ and $H_p^{\mathrm{BM}}(e_\beta;\Q)$, such a map is simply given by a number, which we call the \emph{degree} of the map (with respect to the cells $\alpha$ and $\beta$). For ordinary CW-complexes this is the usual notion of degree. 

For each cell $f_\alpha : [\mathbb D^{n_\alpha}/G_\alpha] \to X$, let $\widetilde f_\alpha : \mathbb D^{n_\alpha} \to X$ denote the composite $\mathbb D^{n_\alpha} \to [\mathbb D^{n_\alpha}/G_\alpha] \to X$. If we choose an orientation of the disk $\mathbb D^{n_\alpha}$, then there are two distinct natural choices of a generator of $H_{n_\alpha}^{\mathrm{BM}}(e_\alpha;\Q)$: if $\mathfrak o \in H_{n_\alpha}(\mathbb D^{n_\alpha},\partial \mathbb D_{n_\alpha};\Z)$ is the fundamental class of $\mathbb D^{n_\alpha}$, then we may take either 
$(\smash{\widetilde f_\alpha})_\ast (\mathfrak o)$ or $\frac 1 {|G_\alpha|} (\smash{\widetilde f_\alpha})_\ast (\mathfrak o)$ as generator.
The former is sometimes called the ``coarse'' or ``crude'' fundamental class, and the latter the ``orbifold fundamental class''. The former has the advantage of being an integral class, and the latter has the advantage that the quotient map $\mathbb D^{n_\alpha} \to [\mathbb D^{n_\alpha}/G_\alpha]$ has degree $\vert G_\alpha|$, as one would expect by counting the cardinality of fibers. In computing degrees in this section, we will work with the crude fundamental class throughout. This is to make the degrees coincide with those calculated by Bianchi;  the rest of the arguments are insensitive to the precise choice of generators.

\subsection{Cell decomposition of the local moduli space} We now proceed to analyze the moduli space $\Bra_{\R^2}^d$ in more detail. The first step is to construct a deformation retraction of $\Bra^d_{\R^2}$ down to a subset of ``local'' branched covers around a point. It will be more convenient to replace $\R^2$ with the open unit disk $D$. 

\begin{defn}Let $\Bra^{d,\mathsf{loc}}_{D} \subset \Bra^d_{D}$ be the closed substack parametrizing covers all of whose components are disks. 

\end{defn}

Recall from \cref{subsec:points over C} the definition of $N(\pi)$ for $\pi\in\fS_d$.

\begin{lem}\label{lem: characterization of Bra loc}
Suppose $f \in \Bra^d_D$ corresponds to a branched cover with $r$ connected components and branch points $z_1, \dotsc, z_k \in D$. Let $\partial D$ denote a loop near the boundary of $D$, encircling all the branch points. Suppose that $f$ has local monodromy $\sigma_1, \dotsc, \sigma_k \in \fS_d$, for a choice of trivialization over some $y \in \partial D$ and paths from $y$ to $z_i$, such that the boundary monodromy $\sigma_\partial$ equals $\sigma_1 \dotsm \sigma_k$. Then the following are equivalent:
\begin{enumerate}
    \item $f \in \Bra^{d,\mathsf{loc}}_D$.
    \item $N(\sigma_1) + \dotsb + N(\sigma_k) = d - r$.
    \item $N(\sigma_1) + \dotsb + N(\sigma_k) = N(\sigma_\partial)$.
\end{enumerate}
\end{lem}

\begin{proof}
Since the monodromy action cannot permute points of different connected components, and $N$ is additive for permutations with disjoint support, we may assume $r = 1$.
So suppose $f \from \Sigma \to D$ with $\Sigma$ connected and note that (1) is equivalent to $\Sigma$ being a disk. 
Now the Riemann--Hurwitz formula says $\chi(\Sigma) + N(\sigma_1) + \dotsb + N(\sigma_k) = d$, so (1) $\iff \chi(\Sigma) = 1 \iff$ (2).

Since $D$ is simply connected, the entire monodromy group is generated by $\{\sigma_1, \dots, \sigma_k\}$. Since fewer than $d-1$ transpositions cannot act transitively on $d$ elements, $N(\sigma_1) + \dotsb + N(\sigma_k) \ge d - 1$. 
On the other hand, the boundary components of $\Sigma$ correspond to cycles in $\sigma_\partial$, so $N(\sigma_\partial) \le d - 1$ with equality holding if and only if $\Sigma$ has a single boundary component.
So if (3) holds then $N(\sigma_\partial) \le d - 1 \le N(\sigma_1) + \dotsb + N(\sigma_k)$ implies that (2) holds. 
Conversely, if (1) and (2) hold then both sides of (3) are equal to $d-1$.\end{proof}

\begin{rem}Since $N$ is subadditive and conjugation invariant, condition (3) in \cref{lem: characterization of Bra loc} implies that $N$ is additive on $\{\sigma_i\}$.\end{rem}

\begin{rem}Let $D_t$ be the open disk of radius $t$ around the origin. A consequence of \cref{lem: characterization of Bra loc} is that if $f \colon E \to D$ is a branched cover, and $f^{-1}(D_t)$ is a union of disks for some $0<t<1$, then also $f^{-1}(D_s)$ is a union of disks for all $0<s<t$.
\end{rem}
\begin{prop}
    The inclusion $\Bra^{d,\mathsf{loc}}_{D} \subset \Bra^d_{D}$ is a deformation retract.
\end{prop}

\begin{proof}Define $\xi : \Bra^d_D \to (0,1]$ by assigning to a branched cover $f:E \to D$ the real number 
$$\sup\, \{t \leq 1 : f^{-1}(D_t) \text{ is a disjoint union of disks}\}. $$
 Then $\xi$ is continuous. Indeed, this may be checked \'etale locally, and hence it suffices to show that $\xi$ restricts to a continuous map on $\Bra^d_{D \rel y}$ for any point $P \in D$. At this point one can just use the explicit basis for the topology of $\Bra^d_{D \rel y}$ (\cref{sec:def of Bra}).  But now the map $\Bra^d_D \to \Bra^d_D$, assigning to a branched cover $f:E \to D$ the branched cover $f^{-1}(D_{\xi(f)}) \to D_{\xi(f)}$, rescaled by a factor $1/\xi(f)$, deformation retracts $\Bra^d_D$ down to  $\Bra^{d,\mathsf{loc}}_{D} $.\end{proof}

Let $\fP$ be the set of all partitions of $d$ (i.e.\ conjugacy classes in $\mathfrak S_d$). 
Each point $a \in \Bra^{d,\mathsf{loc}}$ has a well-defined \emph{boundary monodromy type} $\omega(a) \in \fP$: if $x$ corresponds to a branched cover $E \to D$, then $\omega(a)$ is the monodromy of the cover around a disk in $D$ containing all branch points; since this is an unbased loop and we have not fixed a trivialization anywhere, the monodromy is only well-defined as a conjugacy class in $\fS_d$. Note that the function $\omega$ is discontinuous if $\fP$ is given the discrete topology --- indeed, $\Bra^{d,\mathsf{loc}}_{D}$ is connected. But every $a \in \Bra^{d,\mathsf{loc}}_D$ has a neighborhood on which $\omega(a')$ is a \emph{refinement} of $\omega(a)$. That is, $\omega$ is continuous once we equip $\fP$ with the Alexandroff topology associated with the refinement order. (As in the preceding proof, this can be verified \'etale locally.) Moreover, $\omega$ corresponds precisely to a \emph{stratification} of $\Bra^{d,\mathsf{loc}}_{D}$, with strata $e_\lambda$ indexed by elements $\lambda \in \fP$.

\begin{thm}\label{orbi-cell-structure}
The space $\Bra^{d,\mathsf{loc}}_D$ admits the structure of a CW-orbispace, with open cells $e_\lambda$ for $\lambda \in \fP$. Each cell $e_\lambda$ is orientable, of real dimension $2 \cdot N(\lambda)$.
\end{thm}

\begin{proof}We need to check that the strata $e_\lambda$ are orbi-cells, and construct the attaching maps; the rest is clear. Let $\mathbb E$ be the closed disk of radius $2$ in $\R^2$. Since $\Bra^d_{\mathbb E} \cong [\Bra^d_{\mathbb E \rel y}/\fS_d]$, where $y \in \partial \mathbb E$ is any point, the connected components of $\Bra^d_{\mathbb E}$ are the same as $\fS_d$-orbits of connected components of $\Bra^d_{\mathbb E \rel y}$. Connected components of $\Bra^d_{\mathbb E \rel y}$ were described in \cref{rem: connected components}. Write $B_\lambda$ for the unique component of $\Bra^d_{\mathbb E}$ parametrizing branched covers with boundary monodromy type $\lambda$, all of whose components are disks.

If $\lambda = (d)$ is the partition with a single block of size $d$, then $B_\lambda \cong [\C^{d-1}/\mu_d]$, as follows from \cref{depressed-polynomials,prop:algebraic vs topological branched covers}. Now if $\lambda$ has $a_k$ blocks of each size $k$, then considering one connected component of the covering at a time, we may write similarly
$$ B_\lambda \cong \prod_k \mathrm{Sym}^{a_k} [\C^{k-1}/\mu_k] \cong \prod_k [\C^{a_k(k-1)} / ({\mu}_k^{a_k} \rtimes \fS_{a_k})].$$
In particular, $B_\lambda$ is topologically an open oriented orbi-cell of dimension $2\cdot N(\lambda) = 2\sum_k a_k(k-1).$ 

Consider the map $j^\ast : \Bra^d_{\mathbb E} \to \Bra^d_D$, given by taking a branched cover of $\mathbb E$ and restricting it to the open subset $D$ of $\mathbb E$ (\cref{functorial for open immersions}). Let $U$ be the open substack of $\Bra^d_{\mathbb E}$ defined by the condition that all branch points lie in $D$. Then $B_\lambda \cap U$ maps homeomorphically to $e_\lambda$ under $j^\ast$.  Since $U$ is homeomorphic to $\Bra^d_{\mathbb E}$ by rescaling by a factor $1/2$, we see also that $B_\lambda \cap U$ is homeomorphic to $B_\lambda$, so that $e_\lambda$ is indeed an oriented orbicell of the correct dimension for all $\lambda$. But the map $j^\ast $ also furnishes attaching maps for all the cells. Indeed, $B_\lambda \cap U$ is an open orbi-ball in $B_\lambda$, and its closure is a closed orbi-ball. The restriction of $j^\ast $ to this closed ball is an attaching map for the cell $e_\lambda \subset \Bra^{d,\mathsf{loc}}_D$.  
%
%
%
\end{proof}

\begin{cor}The homology groups $H_\bullet(\Bra_D^{d};\Q)$ are concentrated in even degrees. If $k=2m$, then $\dim H_k(\Bra_D^{d};\Q) =  \#\{\lambda\in\fP : N(\lambda)=m\}$.
\end{cor}

\begin{prop}\label{prop:stability}The stabilization map $H_k(\Bra_D^d;\Q) \rightarrow H_k(\Bra_D^{d+1};\Q)$ is an isomorphism for $k<d+1$. 
\end{prop}

\begin{proof}
    We work instead with $\Bra_D^{d,\mathsf{loc}} \rightarrow \Bra_D^{d+1,\mathsf{loc}}$. The stabilization map is not representable --- for example, the point of $\Bra_D^{d,\mathsf{loc}}$ corresponding to the trivial covering has isotropy group $\fS_d$, and its image in $\Bra_D^{d,\mathsf{loc}}$ has isotropy group $\fS_{d+1}$. For this reason we will find it more convenient to work with the induced map on coarse spaces. On coarse spaces, the map $|\Bra_D^{d,\mathsf{loc}}| \rightarrow |\Bra_D^{d+1,\mathsf{loc}}|$ carries the cell $|e_\lambda|$ isomorphically onto the cell $|e_{\lambda+1}|$, where for each $\lambda \in \mathfrak P_d$ we denote by $\lambda+1 \in \mathfrak P_{d+1}$ the partition obtained by adding a single entry ``$1$'' to $\lambda$. Thus the map is an isomorphism below degree $\min \{ 2N(\lambda) : \lambda \in \mathfrak P_{d+1}\setminus \fP\}$. This minimum is attained by the partition $(2,2,\ldots)$ if $d+1$ is even, and $(3,2,2,\ldots)$ if $d+1$ is odd, and the minimum is in either case\footnote{In particular, the stable range is in fact one better than stated if $d$ is even.} $N=\lceil(d+1)/2\rceil$.
\end{proof}

\subsection{Calculation of the cup-product}To finish the arguments, we will need to know moreover the cup-product in the rational cohomology ring of $\smash{\Bra^d_{D}}$, or dually, the coproduct on homology. For this too, we will use the space $\smash{\Bra^{d,\mathsf{loc}}_{D}}$ and its CW-structure. To obtain a coproduct on cellular chains, one needs to choose a cellular approximation of the diagonal. Fortunately, the space $\smash{\Bra^{d,\mathsf{loc}}_{D}}$ comes with a geometrically meaningful choice of such a cellular approximation.

\begin{lem}Let $i : D \sqcup D \hookrightarrow D$ be the open embedding given by the two maps $D\to D$ defined by \( z \mapsto \tfrac 1 4 z \pm \tfrac 1 2.\) Then the restriction map $$i^\ast : \Bra_D^d \to \Bra^d_{D \sqcup D} = \Bra^d_D \times \Bra^d_D$$
is homotopic to the diagonal, and restricts to a cellular map $\Bra^{d,\mathsf{loc}}_D\to \Bra^{d,\mathsf{loc}}_D\times \Bra^{d,\mathsf{loc}}_D$.
\end{lem}

\begin{proof}
    If $j: D \hookrightarrow D$ is a self-embedding of the disk as a smaller disk, then the induced restriction map $j^\ast :\Bra^d_D \to \Bra^d_D$ is a homotopic to the identity by a straight-line homotopy. So $i^\ast$ is homotopic to the diagonal, since the projection onto each factor is homotopic to the identity. Moreover, it follows from \cref{lem: characterization of Bra loc} that $i^\ast$ carries $\Bra^{d,\mathsf{loc}}_D$ to $\Bra^{d,\mathsf{loc}}_D\times \Bra^{d,\mathsf{loc}}_D$, and restricts to a cellular map.
\end{proof}

Armed with this explicit cellular approximation, we proceed to compute the coproduct on cellular chains $C_\bullet(\Bra^{d,\mathsf{loc}}_D;\Q)$. 

\begin{lem}    Given $\lambda, \mu, \nu \in \fP$ with $N(\lambda) = N(\mu) + N(\nu)$, the degree of the map $i^*$ with respect to the cells $e_\lambda$ and $e_\mu \times e_\nu$ is the number of ways of writing a representative of $\lambda$ as a product of representatives of $\mu$ and $\nu$.
\end{lem}

\begin{proof}In accordance with the discussion in \cref{subsec: CW}, we consider the zigzag
    $$ e_\lambda \leftarrow (e_\lambda \cap (i^\ast)^{-1}(e_\mu \times e_\nu)) \to (e_\mu \times e_\nu)$$
    and the induced map on Borel--Moore homology. 
    The open cell $e_\lambda$ parametrizes branched covers of $D$ by disjoint unions of disks, with boundary monodromy type $\lambda$. Its open substack $(e_\lambda \cap (i^\ast)^{-1}(e_\mu \times e_\nu))$  parametrizes those covers for which all branch points lie in the two smaller disks $D \sqcup D$ (thought of as a subset of $D$ via the embedding $i$), and such that the boundary monodromy types around the two smaller disks are $\mu$ and $\nu$, respectively. Each connected component of this open substack maps isomorphically onto $e_\mu \times e_\nu$.

    Now the fundamental classes of $e_\lambda$, $e_\mu$ and $ e_\nu$ were defined as the pushforwards of the fundamental classes of their respective universal covers $\smash{\widetilde e_\lambda}$, $\smash{\widetilde e_\mu}$ and $\smash{\widetilde  e_\nu}$. We can identify $\smash{\widetilde e_\lambda}$ with the space parametrizing in addition the datum of a trivialization of the cover along a tangential basepoint, such that the boundary monodromy is given by $\pi \in \fS_d$, where $\pi$ is a fixed representative of the conjugacy class $\lambda.$ But now the inverse image of $e_\lambda \cap (i^\ast)^{-1}(e_\mu \times e_\nu)$ inside the disk $\smash{\widetilde e_\lambda}$ is simply a disjoint union of smaller disks, one for each factorization of $\pi$ as a product of a permutation of cycle-type $\mu$ and a permutation of cycle-type $\nu$.  The result follows. 
\end{proof}

\begin{prop}\label{prop:polynomial ring}
    The rational cohomology ring $H^\bullet(\Bra^d_D; \Q)$ agrees in degrees $\leq d $ with a polynomial algebra with a single generator in each positive even degree. 
\end{prop}

\begin{proof}We dualize the structure coefficients for the homology coproduct determined in the previous lemma. Then the cohomology ring has a basis of classes $t_\mu$ for $\mu \in \fP$, with $|t_\mu|=2N(\mu)$, and $t_\mu \cdot t_\nu$ is a linear combination of classes $t_\lambda$ where $\lambda$ ranges over possible cycle-types of permutations which can be obtained by taking the product of a permutation of cycle-type $\mu$ with one of cycle-type $\nu$, and where $N(\lambda)=N(\mu)+N(\nu)$.

If two partitions $\mu$ and $\nu$ admit cycle-type representatives $\pi$ and $\sigma$ with disjoint support, then we define $\mu \cup \nu$ to be the cycle-type of $\pi \circ \sigma$. When this is the case, one sees that 
$$ t_\mu t_\nu = (\text{nonzero constant})\cdot t_{\mu \cup \nu} + \text{ lower order terms, }$$
where the lower order terms are a linear combination of $t_\lambda$, such that $\lambda$ contains ``$1$'' with strictly greater multiplicity than $\mu \cup \nu$. But now if $N(\mu)+N(\nu) \leq  d/2$, then it will always be the case that $\mu$ and $\nu$ admit representatives with disjoint support. It follows that up to this degree, we obtain a polynomial ring on the generators $t_k := t_{(k,1^{d-k})}$, for $k>1$. Indeed, any partition $\lambda = (\lambda_1,\ldots,\lambda_r)$ satisfies $\lambda = \mu_1 \cup \ldots \cup \mu_r$, with $\mu_i = (\lambda_i,1,1,\ldots)$ a hook-shape (i.e.\ a partition with at most one part bigger than $1$), and hook-shapes are the unique $\cup$-irreducible partitions. 
\end{proof}

\begin{prop}\label{prop:rat homotopy}The rational homotopy groups $\pi_k^\Q\Bra_D^d$, where $0<k<d$, are given by 
$$\pi_k^\Q\Bra_D^d \cong \begin{cases}
    \Q & k \text { even,} \\ 0 & k \text{ odd.}
\end{cases}$$
\end{prop}

\begin{proof}Consider the map $$\Bra_D^d \to \prod_{k=1}^{d} K(\Q,2k)$$
which classifies the cohomology classes $t_k$ described in the previous proof. This map induces an isomorphism in rational cohomology up to degree $d$ by \cref{prop:polynomial ring}. To deduce the result from the rational Whitehead and Hurewicz theorems, we only need to know that $\Bra_D^d$ is simply connected. But it is easy to see that $\Bra_D^d$ is connected, and \cref{thm:delooping} implies that also its loop space is connected.
\end{proof}

\begin{thm}
    The Mumford conjecture is true.
\end{thm}

\begin{proof}By \cref{prop:Mg to delooping,prop:infinite d to finite d} we know that 
$$ \varinjlim_g H_\bullet(M_g^1;\Z) = \varinjlim_d H_\bullet(\Omega_0 B \Bra^d_{\square \rel \partial};\Z).$$ 
By \cref{thm:delooping} we also have 
    $$\Omega_0 B \Bra^d_{\square \rel \partial}\simeq \Omega^2_0 \Bra^d_D$$
    for all $d$. By \cref{prop:rat homotopy} we know the rational homotopy type of $\Omega^2_0 \Bra^d_D$ in a range, and in particular that its rational cohomology ring is a polynomial algebra with a single generator in each positive even degree. By \cref{prop:stability} the stabilization maps with respect to $d$ are isomorphisms in a range. The result follows.
\end{proof}





\section{Comparison with Bianchi's papers}
\label{sec:bianchi-comparison}

Let us briefly comment on how the argument presented here compares to Bianchi's. In particular, we indicate for each section of the present paper, which parts of Bianchi's papers play the same role in the global structure of the proof.

\cref{sec:1} corresponds to \cite[Sections~2--5]{bianchi4}. The main goal is to show that a moduli space of branched covers of the disk is a classifying space for $\Mod_g^1$. Our argument does not pass through B\"odigheimer's slit configuration space \cite{bodigheimer}. The space $\M_0^\fr(\mathbb P^1,(d))_\alpha \cong \mathbb A^{d-1}$ appearing in \cref{depressed-polynomials} plays a key role also in Bianchi's arguments; he denotes it $\mathfrak{NMonPol}_d$ \cite[Definition~3.6]{bianchi4}. The fact that it is topologically a cell is ultimately what leads to the CW-decomposition in \cref{sec:5}, as well as Bianchi's calculation of the rational cohomology ring of the delooping.

In \cref{sec:2}, the goal is firstly to define the moduli space of branched covers of the closed square, and secondly to define the spaces which will form $1$-fold and $2$-fold deloopings, given by branched covers of a half-open or open square. 

For the first part, Bianchi defines the moduli spaces of branched covers explicitly as configuration spaces of points, whose points are decorated with local monodromy data. For this purpose he develops a formalism of \emph{partially multiplicative quandles}, and defines appropriate configuration spaces of points where the points have a decoration in a partially multiplicative quandle. The quandle structure encodes that the local monodromy around a branch point is not completely well-defined as an element of the symmetric group, and in particular how local monodromy changes when two branch points rotate around each other. The partial multiplication records how branch points may collide. This is carried out in \cite{bianchi1}. By contrast, we directly topologize the set of isomorphism classes of branched covers with a trivialization at some boundary point.

For the second part, Bianchi does not construct the delooping as a space of decorated configurations where points can vanish along the boundary of the disk. Indeed, this is not really possible unless one works with stacks. Instead, he accomplishes the same effect by remembering the points which lie on the boundary, but allowing them to collide freely. The resulting space therefore has a little copy of a Ran space \cite[Section~5.5.1]{luriehigheralgebra} tacked on along the boundary, so to speak, and the contractibility of the Ran space implies that this has the same effect up to homotopy as completely forgetting the points. This is carried out in \cite{bianchi2}. 

\cref{sec:3} corresponds to \cite[Section~6]{bianchi4}. The main points here are that stabilizing by gluing together branched covers corresponds to Harer stability map (our \cref{prop: good components} and the diagram on the top of p.~39 of \cite{bianchi4}), and that one can ``interchange limits'' when letting the number of branch points go to infinity and letting the degree of the cover go to infinity (our \cref{prop:Mg to delooping,prop:infinite d to finite d}, Bianchi's discussion of the ``main diagram'' and ``propagators''). 

\cref{sec:4} corresponds to \cite[Sections~2--4]{bianchi3}. Both delooping arguments are philosophically similar and follow the proof strategy of \cite{hatcher-madsenweiss}. 

\cref{sec:5} corresponds to \cite[Sections~5--6]{bianchi3}. Our deformation retraction of $\Bra^d_D$ onto $\Bra_D^{d,\loc}$ is the analogue of Bianchi's deformation retraction of $\mathrm{Hur}(\mathcal R,\partial;\mathcal Q,G)_{0;\boldsymbol{1}}$ onto $\mathbb B(\mathcal Q_+,G)$, and the skeletal filtration of $\Bra_D^{d,\loc}$ corresponds to the  norm filtration of $\mathbb B(\mathcal Q_+,G)$. The fact that we construct a decomposition of the moduli space into orbicells, corresponds in Bianchi's paper to the fact that each layer of the norm filtration $\mathfrak F_\nu \mathbb B(\mathcal Q_+,G)$ is fibered over a model of $BG$ with manifold fibers.

\bibliographystyle{alpha}
\bibliography{database}

\end{document}